\documentclass[aap]{imsart}

\RequirePackage{amsthm,amsmath,amsfonts,amssymb}
\RequirePackage[numbers]{natbib}
\RequirePackage[colorlinks,citecolor=blue,urlcolor=blue]{hyperref}
\RequirePackage{graphicx}
\usepackage{graphicx}
\usepackage{float}
\usepackage{soul}
\usepackage{subfigure}
\startlocaldefs

\numberwithin{equation}{section}
\theoremstyle{plain}

\newtheorem{theorem}{Theorem}[section]
\newtheorem{lemma}{Lemma}[section]

\newtheorem{corol}{Corollary}[section]
\theoremstyle{remark}

\newtheorem{remark}{Remark}[section]

\newcommand{\CF}{{\cal F}}

\newcommand{\CR}{\mathbb{R}}

\newcommand{\CH}{\mathcal{H}}
\newcommand{\w}{\omega}
\newcommand{\CW}{\Omega}

\newcommand{\lt}{\left}
\newcommand{\rt}{\right}
\newcommand{\os}{\overline{\sigma}}

\newcommand{\om}{\overline{\mu}}

\newcommand{\us}{\underline{\sigma}}
\newcommand{\um}{\underline{\mu}}

\newcommand{\ep}{\varepsilon}

\newcommand{\ph}{\varphi}

\newcommand{\ra}{\rightarrow}
  \def\esssup{\mathop {\rm ess\,sup}}
  \def\essinf{\mathop {\rm ess\,inf}}
      
   \def\be{\begin{equation}} 
      \def\ee{\end{equation}} 
      \def\beqn{\begin{eqnarray}} 
      \def\eeqn{\end{eqnarray}} 
      \def\beq{\begin{eqnarray*}} 
      \def\eeq{\end{eqnarray*}}
      
        \def\rf#1{\mbox{$(\ref{#1})$}}
             \def\nn{\nonumber}

\endlocaldefs

\begin{document}

\begin{frontmatter}
\title{Strategy-Driven Limit Theorems Associated Bandit Problems\thanksref{T1}}
\runtitle{Strategy-Driven Limit Theorems}
\thankstext{T1}{We thank Larry Epstein, Shige Peng, Xiaodong Yan  and Zhaoang Zhang for valuable discussions.}

\begin{aug}
\author[A]{\fnms{Zengjing} \snm{Chen}\ead[label=e1,mark]{zjchen@sdu.edu.cn}},
\author[B]{\fnms{Shui} \snm{Feng}\ead[label=e2]{shuifeng@mcmaster.ca}}
\and
\author[A]{\fnms{Guodong} \snm{Zhang}\ead[label=e3,mark]{zhang\_gd@sdu.edu.cn}}
\address[A]{School of Mathematics,
Shandong University,
\printead{e1,e3}}

\address[B]{Department of Mathematics and Statistics,
McMaster University,
\printead{e2}}
\end{aug}

\begin{abstract}
Motivated by the study of asymptotic behaviour of the bandit problems, we obtain several  strategy-driven limit theorems including the law of large numbers, the  large deviation principle, and the  central limit theorem.  Different from the classical limit theorems, we develop sampling strategy-driven limit theorems that generate the maximum  or minimum average reward. The law of large numbers identifies all possible limits that are achievable under various strategies. The large deviation principle provides the maximum decay probabilities for deviations from the limiting domain.  To describe the fluctuations around averages, we obtain  strategy-driven  central limit theorems under optimal strategies.  The limits in these  theorem are identified explicitly, and depend heavily on the structure of the events or the integrating functions and  strategies.  This demonstrates the key signature of the learning structure.  Our results can be used to estimate the maximal (minimal) rewards, and  to identify the conditions of avoiding the Parrondo's paradox in the two-armed bandit problem. It also lays the theoretical foundation for statistical inference in determining the arm that offers the higher mean reward.
\end{abstract}



\end{frontmatter}

\section{Introduction}\label{introduction}

The bandit problem is a special type of sequential random sampling (see \cite{bradt,feldman,robbins,thompson}).  The prototype for the classical ``multi armed bandit'' (MAB)  is a slot machine with finite number of arms.   When an arm is pulled, the player will receive a reward according  to a probability distribution  for that arm.  The probability distributions of the rewards for different arms are independent, and but unknown.  If the number of arms is two, we will call the problem the two-armed bandit or TAB problem. For ease of presentation, we focus  on TAB problem in this paper. Generalizations to MAB problem can be done with minor adjustment.

Since the players and Casino owners have opposite goals,  game fairness becomes a central issue . It is thus natural for both player and machine designer to consider the following questions.
 \begin{description}
\item{(a)} {\bf Parameter Estimation:} What sampling strategies or a sequence of arm pulls can produce the greatest possible expected average value of the sum of rewards in the long run or as the number of plays increase? Since the expected rewards for different arms are unknown parameters,  one would need to develop tools for the estimation of the maximal expected rewards (or the minimal expected rewards) of all arms.
   Early studies on this can be found in Robbins \cite{robbins}.

\item{(b)} {\bf Hypothesis Testing:} Assuming that the estimate has been found for the maximal/minimal expected rewards. How to identify the arm with the maximal/minimal expected rewards? One solution to the problem is to perform a hypothesis test by identifying a test statistic and its asymptotic distribution.   Whittle \cite {whittle} raised the question without providing answers.

\item{(c)} {\bf Parrondo's Paradox:} The Parrondo's paradox devised by physicist Parrondo \cite{Parrondo}, corresponds to a counterintuitive  phenomenon where a combination of two losing strategies leads to a winning one. The phenomenon can be proved  to occur in the antique Mills Futurity slot machine (see for example \cite{Ethier-Lee} in details). It is clearly in the interests of both parties to determine whether the paradox occurs and what the long run outcomes are.
 \end{description}

 Motivated by the study of the asymptotic behaviour of these questions, we develop a framework of strategy-driven limit theorem and terminology for the study of TAB problem. As applications, we shall use our strategy-driven limit theorem to answer the above questions in Section \ref{application}.

The first known paper on bandits was Thompson \cite{thompson}. The motivation for the study came from clinical trials where one would need to select one treatment from several treatments to be used for the next patient based on the performances already observed. The mathematical formulation was in the Bayesian framework. Bradt et al. \cite {bradt} considered the two-armed bandit problem, in which one knew both the maximum mean and the minimum mean of rewards, but a prior distribution was assigned to the mean for each arm.  Bellman \cite{bellman}  referred to this problem as the two-machine problem. In the seminal paper \cite{Gittins1979}, Gittins introduced the Gittins index and obtained the optimal solution for a class of Markovian bandits.  The restless bandits, a more general Makovian bandits, was introduced later in Whittle \cite{whittle88}. All these models are special cases of Bayesian bandits.  The monograph \cite{Gittins2011} provides a comprehensive coverage on the development of bandit problem in Bayesian framework.

In his seminal work, Robbins \cite{robbins} formulated the TAB problem in a  frequentist setting. He established a strong law of large numbers to investigate the optimal strategies of TAB problem. Under this formulation, Lai and Robbins \cite{Lai-Robbins85} proposed an important concept ``regret'' to study TAB problem, and introduced the technique of upper confidence bounds (UCB) for the asymptotic analysis of regret.
By modifying different components of the TAB problem, one is led to numerous other generalizations. Examples include but not limited to the non-i.i.d. rewards \cite{Per-Rig13}, the combinatorial bandit problem \cite{Chen-Wang-Yuan13}, and contextual multi-armed bandit \cite{Chen-Lu-Song21}.
The bandit problem also finds applications in a wide range of areas including clinical trials, biological modelling, data processing, internet, and machine learning (see for example \cite{Gittins2011,jacko,Sutton2018,thompson}).  For a comprehensive coverage of the topics, one could refer to \cite{Slivkins19}, \cite{Lar-Sze20}, and the references therein.

In this paper, 
we first establish the strategy-driven weak and strong law of large numbers. The strong law of large numbers generalizes the result in \cite{robbins}.


Our second strategy-driven result is called strategic central limit theorem.
The statistic in our strategic central limit theorem  is different from  the classical central limit theorem in which the individual's decision, effort, strategy or experience does not play any role what so ever. Since both the sample mean and sample fluctuation depend on the sampling strategy,
our result will be in terms of some special strategies and the maximal probabilities over all strategies. The limiting distributions, which are explicitly identified, will in general not be normal. Due to the nonlinear nature of the model, the limiting distributions will be set dependent. The strategies that achieve the limits will also be set dependent. Chen et al. \cite{chenepsteinzhang} first applied the nonlinear strategic central limit theorem to the bandit problems, they considered the bandit rewards with uncertain variances but common means, and their results are studied in a nonlinear probability setting with a set of rectangular measures. What they are mainly concerned is how to give a strategy to maximize expected utility when the decision-maker is loss averse.
However, our central limit theorem is mainly focus on the bandit rewards with uncertain means, and these results lay the theoretical foundation for statistical raison estimation and raison testing hypotheses in determining the arm that offers a higher chance of reward. To the best of our knowledge, this is the first result where the test statistic and rejection region are constructed explicitly in the hypothesis test for the TAB problem.

Our third result is the large deviation principle associated with the strategic law of large numbers. It provides more refined information than the corresponding law of large numbers. This is different from the large deviation estimates used in \cite{Lai-Robbins85}. More specifically,  the asymptotically efficient strategy obtained in the paper has a logarithmic growth rate in terms of regret. The constants appearing in the estimates are given by the Kullback-Leibler information, which follows from the large deviation principle for each individual arm. Our large deviation result is for the whole sampling sequence instead of individual arms.

\vspace{1mm}

The layout of the paper is as follows. In Section 2, we introduce the model, the assumptions,  and the notations used throughout the paper. We also present a basic lemma involving the conditional moments of the model.
In Section 3, we present the strategy-driven limit theorems including the law of large numbers, the strategic central limit theorems, and  the large deviation principle. The limiting distribution in the strategic central limit theorem depends strongly on the integrating function and the strategies, which demonstrates the fundamental structural differences from classical central limit theorem.  In Section 4,  we consider the applications of our  strategic limit theorems. All proofs are collected  in Section 5.

\section{Basic Settings}\label{assump}
 Assume that $(\CW,\CF,P)$ is a probability space and  two random variables $W^L$ and  $W^R$ represent the random rewards from arms {\bf L}  and {\bf R} respectively.  Let $\{W_i^L:i\ge1\}$ and $\{W_i^R:i\ge1\}$ denote the sequence of random rewards from arms {\bf L}  and {\bf R}, which are  the independent and identically distributed copies of $W^L$ and $W^R$.
A sampling strategy $\theta$ is usually defined by a sequence of  random variables  $\theta=\{\vartheta_1,\cdots, \vartheta_n,\cdots \}$
where $\vartheta_i=1$ (respectively, $\vartheta_i=2$) means arm {\bf L} (respectively, arm {\bf R}) is selected at round $i$. The reward at round $i$ under the strategy $\theta$ is then given by
\begin{equation}\label{eq-1}
Z_i^{\theta}=\left\{\begin{array}{ll}
W_i^L,& \text{ if }\ \vartheta_i=1,\\
W_i^R, &\text{ if }\ \vartheta_i=2.
\end{array}\right.
\end{equation}
 In the sequel, we assume that $W^L$ and $W^R$ have finite means and variances, which are denoted by
\begin{equation}\label{musigma}
\begin{array}{l}
\mu_L:=E_{P}[W^L], \quad  \sigma_L^2:=\mbox{Var}_{P}\lt[W^L\rt],\\
\mu_R:=E_{P}[W^R], \quad \sigma_R^2:=\mbox{Var}_{P}\lt[W^R\rt].
\end{array}
\end{equation}
Set
\beq
&&\om=\max\{\mu_L,\mu_R\}, \quad   \um=\min\{\mu_L,\mu_R\},\\
&&\os^2=\max\{\sigma_L^2,\sigma_R^2\},\quad \us^2=\min\{\sigma_L^2,\sigma_R^2\}.\eeq
From the lemma below, we can see that  $\om, \os^2$ and $\um, \us^2$ are the upper and lower conditional means and variances of $Z_n^{\theta}$, respectively.

Recall that a sampling strategy $\theta$ is defined by  a sequence of $\{1,2\}$-valued  random variables  $\theta=\{\vartheta_1,\cdots,\vartheta_i,\cdots \}$.
We call a sampling strategy $\theta$ {\it admissible} if $\vartheta_n$ is $\mathcal{H}_{n-1}^{\theta}$-measurable for all $n \geq 1$, where
$$\mathcal{H}_n^{\theta}=\sigma\{Z_1^{\theta},\cdots,Z_n^{\theta}\}\text{ and }\mathcal{H}_0^{\theta}=\{\emptyset,\CW\}.$$
The set $\Theta$ denotes the collection of all admissible sampling strategies.

We end the section with a lemma on conditional moments that will be used repeatedly in the sequel.
\begin{lemma}\label{tab-proper}
The random rewards $\{Z_n^{\theta}: n\geq 1\}$ defined in (\ref{eq-1}) satisfy the followings.

\begin{description}
\item[(1)]  For any  $n\ge1$, we have
\begin{align*}
\esssup\limits_{\theta\in\Theta}E_P[Z_n^{\theta}|\CH^{\theta}_{n-1}]=\om,\ \essinf\limits_{\theta\in\Theta}E_P[Z_n^{\theta}|\CH^{\theta}_{n-1}]=\um,\\
\esssup\limits_{\theta\in\Theta}E_{P}\lt[\lt(Z_n^{\theta}-E_{ P}[Z_n^{\theta}|\CH^{\theta}_{n-1}]\rt)^2|\CH^{\theta}_{n-1}\rt]=\os^2,\\
\essinf\limits_{\theta\in\Theta}E_{P}\lt[\lt(Z_n^{\theta}-E_{ P}[Z_n^{\theta}|\CH^{\theta}_{n-1}]\rt)^2|\CH^{\theta}_{n-1}\rt]=\us^2.
\end{align*}
\item[(2)] For any $\theta\in\Theta$ and $n\ge1$, let $U_{n-1}^{\theta}$ be any $\theta$-dependent (only depend on $(\vartheta_1,\cdots,\vartheta_{n-1})$) and $\CH^{\theta}_{n-1}$-measurable random variable. For any bounded measurable functions $f_0,f_1$ and $f_2$ on $\CR$, let $\psi(x,y)=f_0(x)+f_1(x)y+f_2(x)y^2, (x,y)\in \CR^2$. Then we have\small
\begin{equation*}
\sup_{ \theta\in\Theta}E_{P}\lt[\psi\lt(U_{n-1}^{\theta},Z_n^{\theta}\rt)\rt]=
\sup\limits_{ \theta\in\Theta}E_{P}\lt[\psi_n^L\big(U_{n-1}^{\theta}\big)\vee\psi_n^R\big(U_{n-1}^{\theta}\big)\rt]
\end{equation*}\normalsize
where for $ x\in\CR$,
\beq
\psi^L_n(x)&=&E_{P}[\psi(x,W^L_n)]=f_0(x)+\mu_L\, f_1(x)+(\mu_L^2+\sigma_L^2)\, f_2(x),\\
\psi^R_n(x)
&=&E_{P}[\psi(x,W^R_n)]=f_0(x)+\mu_R\, f_1(x)+(\mu_R^2+\sigma_R^2)\, f_2(x).
\eeq
\end{description}
\end{lemma}

\section{Main Results}\label{main-results} Let $\{Z_i^{\theta}:i\ge1\}$ be defined in (\ref{eq-1}).
For  each $n\geq 1$ the average rewards of the first $n$ rounds under strategy $\theta$ is given by $$S_n^{\theta}=\sum_{i=1}^n Z_i^{\theta}.$$ The main results of this paper deal with  the asymptotic behaviours of $S_n^{\theta}/n$ and associated fluctuations when $n$ tends to infinity. These include the law of large numbers, the strategic central limit theorem, and the large deviation principle.

\subsection{The law of large numbers} Our first result is the law of large numbers.  Since the limiting behaviour of $S_n^{\theta}/n$ strongly depends on the strategies, we establish two kinds of (strong and weak) law of large numbers (LLN).

\begin{theorem}\label{thm-lln}
\begin{description}
\item[(1) \bf Strategic  strong LLN:]   For any $h\in[\um,\om]$ with the representation
    $$h=\gamma\om+(1-\gamma)\um,\quad \gamma\in[0,1],$$
    one can construct  a strategy  $ \theta^{\gamma}$ (shown in Section 5) such that
\begin{equation}\label{lln-4}
\lim\limits_{n\rightarrow\infty} \frac{S_n^{\theta^{\gamma}}}{n}=h, \ P\text{-a.s.}
\end{equation}
   \item[(2) \bf Weak LLN:]
   For any  $\ep>0$,
\begin{equation}\label{lln-1}
\lim\limits_{n\rightarrow\infty}\inf_{\theta\in\Theta}P\lt(\um-\ep<\frac{S_n^{\theta}}{n}<\om+\ep \rt)=1.
\end{equation}
For any   $\ep>0$, $h \in[\um,\om]$,
\begin{equation}\label{lln-2}
\lim\limits_{n\rightarrow\infty}\sup_{\theta\in\Theta}P\lt(\lt|\frac{S_n^{\theta}}{n}-h\rt|<\ep \rt)=1.
\end{equation}
 \end{description}
\end{theorem}

\begin{remark}
 The strong law of large numbers can be applied to estimate the maximal (or minimal) expected rewards of two arms, and the weak law of large numbers will help in identifying conditions when Parrondo's paradox does not hold. The details will be presented  in Section~\ref{application}.
\end{remark}
\subsection{ Strategic  central limit theorem} The second main result is a new central limit theorem.  It identifies the limiting distributions of various fluctuations around $S_n^{\theta}/n$, and provides the theoretical tools for performing hypothesis testing.

We usually characterize the uncertainty of arm returns from two perspectives: mean and variance.  Without loss of generality, we assume that
 both arms have the common variances $\sigma_L^2 =\sigma_R^2>0$ defined in (\ref{musigma}), that is
 \be\label{common-variance}
 \os^2=\us^2=:\sigma^2>0.
 \ee
For the case of different variances,  we can normalize the random rewards of the two arms defined in (\ref{eq-1}) and define
\begin{equation*}
Z_i^{\theta}=\left\{\begin{array}{ll}
\frac{W_i^L}{\sigma_L},& \text{ if }\ \vartheta_i=1,\\
\\
\frac{W_i^R}{\sigma_R}, &\text{ if }\ \vartheta_i=2.
\end{array}\right.
\end{equation*}

Different from the classic central limit theorem, our result depends heavily on the structure of the events or the integrating functions and strategies.

 In this paper, we will focus on symmetric  integrating functions. For any constant $c$ in $\CR$, a function $\ph$ defined on $\CR$ is \emph{symmetric with centre $c$} if  $\ph(x+c)=\ph(-x+c)$ for any $x\in{\CR}.$

 We say a random variable $\eta$  is { \bf Bandit distributed } with parameter $(\alpha,\beta,c)\in \CR^3$ along with a symmetric function $\ph$ with centre $c$ if its density function is $f^{\alpha,\beta,c}$ denoted by
\begin{equation}
f^{\alpha,\beta,c}(y)= \frac{1}{\sqrt{2\pi }}e^{-\frac{(y-\beta)^2-2\alpha (|y-c|-|c-\beta|)+\alpha^2}{2}}-\alpha e^{2\alpha |y-c|}
\Phi(-|c-\beta|-|y-c|-\alpha),\label{proba-density}
\end{equation}
where $\Phi$ is the distribution function of standard normal distribution, and denote it by $ \eta\sim \mathcal{B}(\alpha,\beta,c).$
\begin{remark} Let  $\beta=0$ and $c=0$, the density function of Bandit distribution has the following properties:
\begin{itemize}
 \item If $\alpha <0$, the image of Bandit distribution is spike, referred as a spike distribution.
 \item If $\alpha>0$, the Bandit distribution is similar to two normal distributions {\bf hand in hand,} referred as a binormal distribution.
 \item If $\alpha=0$, the Bandit distribution is degenerated to a standard normal distribution.
\end{itemize}
The density function of a Bandit distribution is shown in the following figures.
 \begin{figure}[H]
  \centering
    \includegraphics[width=4in]{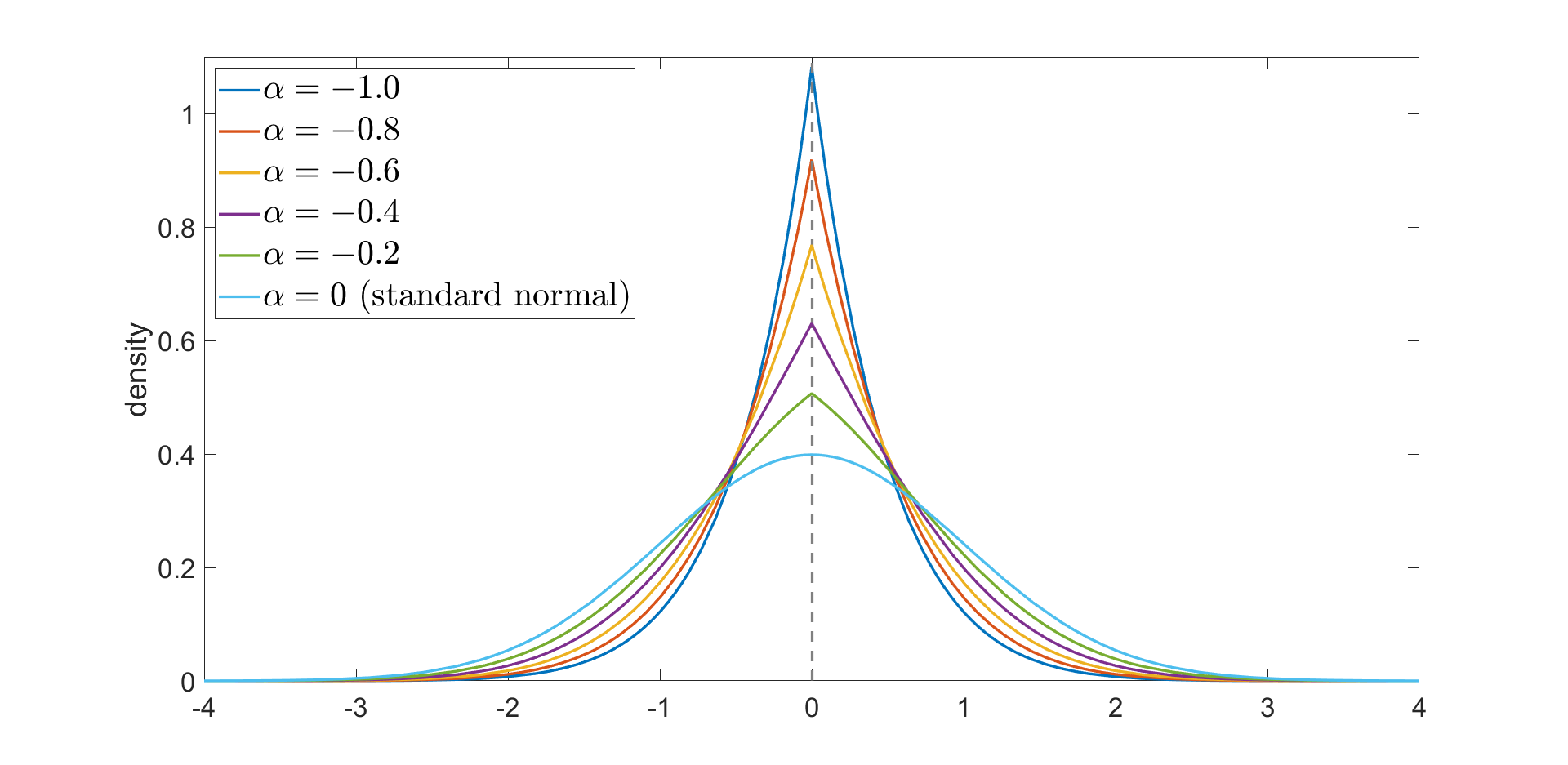}
    \caption{when $\alpha\le0$}
\end{figure}
 \begin{figure}[H]
  \centering
    \includegraphics[width=4in]{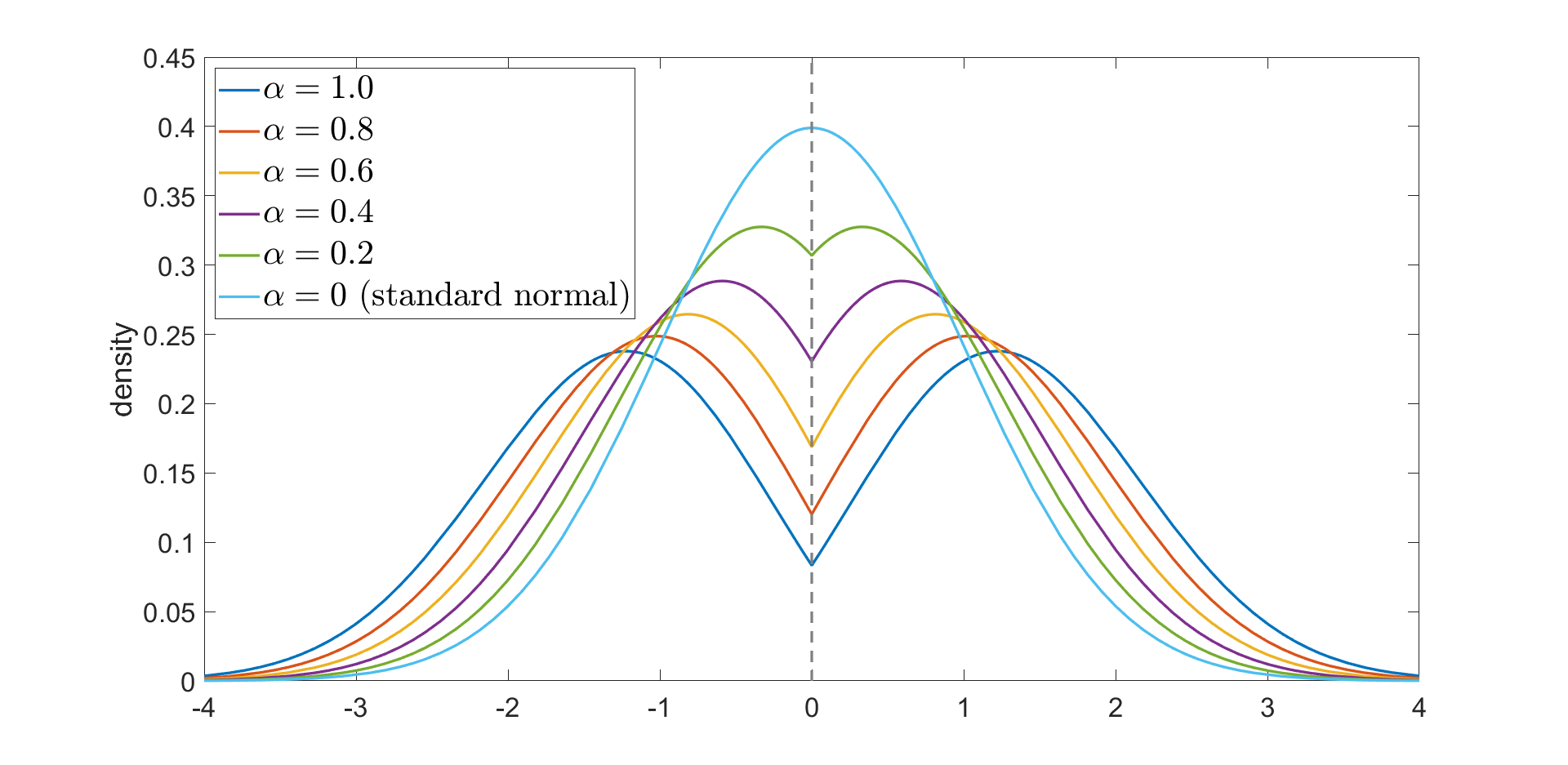}
    \caption{when $\alpha\ge0$}
\end{figure}

The feature of Bandit distribution inspires us to conduct a hypothesis test through a statistic that has asymptotic Bandit distribution.  It turns out to perform better than the classical approach with normal distribution (see Section~\ref{application} for details).
\end{remark}
 We now consider the limit distribution of  the following statistics: For any $\theta\in\Theta $ and $n\in \mathbb{N}^+$, define  for $1\le m\le n$,
 \begin{align}
T_{m,n}^{\theta}=\frac{1}{n}\sum\limits_{i=1}^{m} Z_i^{\theta}+\frac{1}{\sqrt{n}}\sum\limits_{i=1}^{m}\frac{1}{\sigma}\lt(Z_{i}^{\theta}-E_P[Z_i^{\theta}|\CH^{\theta}_{i-1}]\rt),\, \ T_{0,n}^{\theta}=0.\label{TQmn}
\end{align}
where $\{Z_{i}^{\theta}:i\ge1\}$ defined in (\ref{eq-1}).

It is obviously that, for the ``single'' or ``independent'' strategy $\theta=(1,1,\cdots,1,\cdots)$ (or $(2,2,\cdots,2,\cdots)$), which means choosing arm {\bf L} (or {\bf R}) repeatedly regardless of previous outcomes, the returns $Z_n^{\theta}$ will be independent  and identically distributed, and the limit distribution of $T_{n,n}^{\theta}$ will be a normal distribution. However, for some ``switching'' or ``dependent'' strategies, which means one may choose arms depend on the previous outcomes, the distribution of $Z_n^{\theta}$ will also be history dependent, and the limit distribution of $T_{n,n}^{\theta}$ will be more difficult to study, or it may not even exist.

Here we will construct a sequence of strategies $\theta^{n,c}=(\vartheta^{n,c}_1,\cdots,\vartheta^{n,c}_{m},\cdots)$ as follows,
\begin{equation}\label{strategy}
\text{for fixed }n\ge1,c\in\CR,\quad\text{ let }\ \vartheta^{n,c}_{m}=
2-I_{\{T_{m-1,n}^{\theta^{n,c}}\le c-(1-\frac{m-1}{n})\frac{\om+\um}{2}\}},\ \  \text{ for }\  m \ge 1,
\end{equation}
under which the limit distribution of $T_{n,n}^{\theta^{n,c}}$ will be described by a Bandit distribution.

 Immediately, we have the following strategic central limit theorem for  a symmetric function with centre $c$.  An explicit formula for the limit distribution  is given as follow.
\begin{theorem}\label{thm-CLT0}
Assume that the rewards of the two arms have common conditional variance $\sigma^2$ given in (\ref{common-variance}).
Let $\ph\in C(\overline\CR)$ be  a continuous function on $\CR$ with finite limits at $\pm \infty$, and be symmetric with centre $c\in\CR$ and monotone on $(c,\infty)$, then the limit distributions of $\{T_{n,n}^{\theta^{n,c}}\}$ are Bandit distributed. That is
\begin{description}
\item[(1)] Under the hypothesis $(\mu_L,\mu_R)=(\om,\um)$, we have
\begin{equation}
\lim\limits_{n\rightarrow\infty}E_{P}\left[
\varphi\left( T_{n,n}^{\theta^{n,c}}\right)  \right]
=E_P[\ph(\eta_1)],\label{strategy-CLT}
\end{equation}
where $\eta_1\sim \mathcal{B}\lt(\frac{\um-\om}{2},\frac{\om+\um}{2},c\rt)$.
\item[(2)]  Under the hypothesis $(\mu_L,\mu_R)=(\um,\om)$, we have
\begin{equation}
\lim\limits_{n\rightarrow\infty}E_{P}\lt[
\varphi\lt( T_{n,n}^{\theta^{n,c}}\rt)  \rt]
=E_P[\ph(\eta_2)], \label{strategy-CLT2}
\end{equation}
where $\eta_2\sim \mathcal{B}\lt(\frac{\om-\um}{2},\frac{\om+\um}{2},c\rt)$.
\end{description}
\end{theorem}

\begin{remark}
The strategy $\theta^{n,c}$ goes as follows: for the first round we choose arm {\bf L} if $\frac{\om+\um}{2}\le c$, otherwise choose arm {\bf R}, and then obtain the value of statistic $T_{1,n}^{\theta^{n,c}}$; for the $m^{th}\ (m\ge2)$ round we choose arm {\bf L} if $T_{m-1,n}^{\theta^{n,c}}\le c-(1-\frac{m-1}{n})\frac{\om+\um}{2}$,  otherwise choose arm {\bf R}, and then obtain the value of statistic $T_{m,n}^{\theta^{n,c}}$. Because the strategy $\vartheta^{n,c}_m$ at the $m^{th}$ round depends on observation of the  first $m-1$ rounds, our strategy and statistics are reminiscent of the idea of raison.
\end{remark}
The distributions in (\ref{proba-density}) of Bandit distribution are  complex, but  when  $\varphi$ is a indicator function on the interval $[a,b]$, its  probability  on  interval $[a,b]$ is beautiful and easily computing.

The next corollary follows from Theorem~\ref{thm-CLT0} for a indicator function  $\varphi$ on the interval $[a,b]$ and the standard approximation arguments.
\begin{corol}\label{calculate-indi}
For $a<b\in\CR$, we have
\begin{description}
\item[(1)] If $(\mu_L,\mu_R)=(\om,\um)$, then
\begin{align*}
\lim\limits_{n\rightarrow\infty}P\left( a\le T_{n,n}^{\theta^{n,c}}\le b\right)  =&\left\{
\begin{array}
[c]{lc}%
\Phi\left(\om- a\right)  -e^{\frac{(\um-\om)(b-a)}{2}}\; \Phi\left(\om -b\right),  & \ \mbox{ if
}\ a+b\geq \om+\um,\\
\Phi\left( b-\um\right)  -e^{\frac{(\um-\om)(b-a)}{2}}\; \Phi\left( a-\um\right),  & \ \mbox{ if
}a+b<\om+\um,%
\end{array}
\right.
\end{align*}
\item[(2)] If $(\mu_L,\mu_R)=(\um,\om)$, then
\begin{align*}
\lim\limits_{n\rightarrow\infty} P\left( a\le T_{n,n}^{\theta^{n,c}}\le b\right)=&\left\{
\begin{array}
[c]{lc}%
\Phi\left(\um- a\right)  -e^{\frac{(\om-\um)(b-a)}{2}}\; \Phi\left(\um -b\right),  & \ \mbox{ if
}\ a+b\geq \om+\um,\\
\Phi\left( b-\om\right)  -e^{\frac{(\om-\um)(b-a)}{2}}\; \Phi\left( a-\om\right),  & \ \mbox{ if
}a+b<\om+\um,%
\end{array}
\right.
\end{align*}
where  $\Phi$ denotes the  distribution function of standard normal distribution.
\end{description}
\end{corol}
The next theorem shows that under some hypothesis the strategies $\{\theta^{n,c}\}$ will be asymptotically optimal.

\begin{theorem}\label{thm-CLT1}
Assume that the rewards of the two arms have common conditional variance $\sigma^2$ given in (\ref{common-variance}). For any fixed $c\in\CR$, let $\ph$ be as in Theorem \ref{thm-CLT0}, then the strategies $\{\theta^{n,c}\}$ are asymptotically optimal in the following sense.
\begin{description}
\item[(1)] If $\varphi$ is decreasing on $(c,\infty)$ and $(\mu_L,\mu_R)=(\om,\um)$, we have
\begin{equation}
\lim\limits_{n\rightarrow\infty}E_{P}\left[
\varphi\left( T_{n,n}^{\theta^{n,c}}\right)  \right]=\lim\limits_{n\rightarrow\infty}\sup_{\theta\in\Theta}E_{P}\left[
\varphi\left( T_{n,n}^{\theta}\right)  \right]
=E_P[\ph(\eta_1)],\label{CLT}
\end{equation}
where $\eta_1\sim \mathcal{B}\lt(\frac{\um-\om}{2},\frac{\om+\um}{2},c\rt)$.
\item[(2)] If $\varphi$ is increasing on $(c,\infty)$ and $(\mu_L,\mu_R)=(\um,\om)$, we have
\begin{equation}
\lim\limits_{n\rightarrow\infty}E_{P}\left[
\varphi\left( T_{n,n}^{\theta^{n,c}}\right)  \right]=\lim\limits_{n\rightarrow\infty}\sup_{\theta\in\Theta}E_{P}\lt[
\varphi\lt( T_{n,n}^{\theta}\rt)  \rt]
=E_P[\ph(\eta_2)], \label{CLT2}
\end{equation}
where $\eta_2\sim \mathcal{B}\lt(\frac{\om-\um}{2},\frac{\om+\um}{2},c\rt)$.
\end{description}
\end{theorem}
\begin{remark} Without the assumptions $(\mu_L,\mu_R)=(\om,\um)$ and $(\mu_L,\mu_R)=(\um,\om)$, the second equality in (\ref{CLT}) and (\ref{CLT2}) still holds.
\end{remark}

Under the (order) hypothesis $(\mu_L,\mu_R)=(\om,\um)$, for any strategy $\theta\in\Theta$, the conditional mean $E_P[Z_{i}^{\theta}|\CH^{\theta}_{i-1}]$ in (\ref{TQmn}) is $\CH^{\theta}_{i-1}$-measurable and can be expressed through $\theta$ explicitly as
\begin{equation}\label{mumn}
E_P[Z_{i}^{\theta}|\CH^{\theta}_{i-1}]=\om I_{\{\vartheta_i=1\}}+\um I_{\{\vartheta_i=2\}}=:\mu_i^{\theta}.
\end{equation}
In the final of this section, we will consider a test statistic $\hat{T}_{n,n}^{\theta}$, which is  $T_{n,n}^{\theta}$ in (\ref{TQmn}) with $E_P[Z_{i}^{\theta}|\CH^{\theta}_{i-1}]$ replaced by $\mu_i^{\theta}$, that is,
\begin{align}
\hat{T}_{m,n}^{\theta}=\frac{1}{n}\sum\limits_{i=1}^{m} Z_i^{\theta}+\frac{1}{\sqrt{n}}\sum\limits_{i=1}^{m}\frac{1}{\sigma}\lt(Z_i^{\theta}-\mu_{i}^{\theta}\rt),\, \ 0\le m\le n.\label{Tmn}
\end{align}
Also the strategy $\theta^{n,c}$ in (\ref{strategy}) can be rewrite in the form of $\hat{T}_{m,n}^{\theta}$, we denote it by $\hat{\theta}^{n,c}=(\hat{\vartheta}^{n,c}_1,\cdots,\hat{\vartheta}^{n,c}_{m},\cdots)$
as follows,
\begin{equation}\label{strategy-hypo}
\hat{\vartheta}^{n,c}_{m}=
2-I_{\{\hat{T}_{m-1,n}^{\hat{\theta}^{n,c}}\le c-(1-\frac{m-1}{n})\frac{\om+\um}{2}\}},\ \  \text{ for }\  m \ge 1.
\end{equation}

Combine with Theorem \ref{thm-CLT0}, we will show the limit distribution of $\hat{T}_{n,n}^{\hat{\theta}^{n,c}}$ in the following corollary, which can be used to conduct the hypothesis testing in Section~\ref{section-test}.


\begin{corol}\label{sym-max}
Let  $c\in\CR$, $\ph$ be as in Theorem \ref{thm-CLT0}.
\begin{description}

\item[(1)]  If $(\mu_L,\mu_R)=(\om,\um)$, then
\begin{equation}\label{sym-clt1}
\lim\limits_{n\rightarrow\infty}E_{P}\left[
\ph\lt(  \hat{T}_{n,n}^{\hat{\theta}^{n,c}}\rt)  \right]
=E_P[\ph(\eta_1)],
\end{equation}
where $\eta_1\sim \mathcal{B}\lt(\frac{\um-\om}{2},\frac{\om+\um}{2},c\rt)$.

Furthermore, if $\om+\um=0$, for any $a>0$, let $c=0,$  we have
\begin{align}\label{sym-indi1}
\lim\limits_{n\rightarrow\infty} P\left(| \hat{T}_{n,n}^{\hat{\theta}^{n,c}}|\le a\right)
=&\Phi\left(\om+a\right)  -e^{-2\om a}\; \Phi\left(\om -a\right).
\end{align}
\item[(2)]  If $(\mu_L,\mu_R)=(\um,\om)$, then
\begin{equation}\label{clt1-H1}
\lim\limits_{n\rightarrow\infty}\lt\{E_{P}\left[
\ph\lt( \hat{T}_{n,n}^{\hat{\theta}^{n,c}}\rt)  \right]
-E_P[\ph(\hat{\sigma}\hat{\eta}_n)]\rt\}=0,
\end{equation}
where $\hat{\eta}_n\sim\mathcal{B}\lt(\hat{\alpha}_n,\frac{\om+\um}{2\hat{\sigma}},\frac{c}{\hat{\sigma}}\rt)$ and $\hat{\alpha}_n=(1+2\sqrt{n}/\sigma)\frac{\om-\um}{2}$,  $\hat{\sigma}=\sqrt{1+(\om-\um)^2/\sigma^2}$.

Furthermore, if $\om+\um=0$, for any $a>0$, let $c=0$  we have
\begin{equation}\label{sym-indi2}
\lim\limits_{n\rightarrow\infty}\lt\{ P\left(| \hat{T}_{n,n}^{\hat{\theta}^{n,c}}|\le a\right)
-\lt[\Phi\left(-\hat{\alpha}_n+\frac{a}{\hat{\sigma}}\right)  -e^{\frac{2\hat{\alpha}_n a}{\hat{\sigma}}}\Phi\left( -\hat{\alpha}_n-\frac{a}{\hat{\sigma}}\right)\rt]\rt\}=0
\end{equation}
\end{description}
\end{corol}
\begin{remark}

In Section~\ref{section-test}, we will show that the statistic constructed through the strategy $\hat{\theta}^{n,c}$ and the Bandit distribution performs better than the statistic constructed through ``single'' strategy $\theta=(1,1,\cdots,1,\cdots)$ (or $(2,2,\cdots,2,\cdots)$) and normal distribution in the hypothesis testing.
\end{remark}
\subsection{Large deviation principle}
 The law of large numbers identify $[\um, \om]$ as the limiting interval. The probabilities under all strategies will thus be asymptotically small out side the interval.  Our next result gives the estimates on the maximum decay rate of all strategies outside the interval.  Set
 \[
 \Lambda_{\mu_L}(\lambda)= \log E[e^{\lambda W^L}], \  \Lambda_{\mu_R}(\lambda)= \log E[e^{\lambda W^R}], \ \text{for }\lambda\in\CR. \]
We  assume that
 \be\label{mgf2}
   \max\{\Lambda_{\mu_L}(\lambda),\Lambda_{\mu_R}(\lambda)\} =\left\{\begin{array}{ll}
  \Lambda_{\om}(\lambda),& \mbox{ if } \lambda \geq 0,\\
    \Lambda_{\um}(\lambda), & \mbox{ if } \lambda <0.
 \end{array}
 \right.
 \ee



Then we have the following large deviation principle.

 \begin{theorem}\label{LDP}
 For any $n\geq 1$, set
   $$\nu_n(A)=\sup_{\theta\in\Theta}P\lt(\frac{S_n^{\theta}}{n} \in A\rt),\quad A\in \mathcal{B}(\CR),$$
where $ \mathcal{B}(\CR)$ is the Borel $\sigma$-algebra on $\CR$. Then, under the assumption \rf{mgf2}, the family $\{\nu_n: n=1,2,\ldots\}$ satisfies a large deviation principle on $\mathbb{R}$ with speed $n$ and rate function
 \be\label{ratefunction}
 I(x)=\left\{\begin{array}{ll}
  \Lambda^{\ast}_{\om}(x),& \mbox{ if } x>\om,\\
  \Lambda^{\ast}_{\um}(x),& \mbox{ if } x< \um,\\
    0, & \mbox{ if } x\in [\um,  \om].
 \end{array}
 \right.
 \ee
 where for $ x \in \mathbb{R}$
 \be
 \begin{array}{l}
 \Lambda^{\ast}_{\om}(x)=\sup\limits_{\lambda \in\CR}\{\lambda x- \Lambda_{\om}(\lambda)\}, \\
 \Lambda^{\ast}_{\um}(x)=\sup\limits_{\lambda \in\CR}\{\lambda x- \Lambda_{\um}(\lambda)\}.
 \end{array}\label{Lambda-star-om-um}
 \ee

 Namely,  for any closed set $F$ in $\mathbb{R}$  we have the upper estimate
 \be\label{ldp1}
 \limsup_{n\ra \infty}\frac{1}{n}\log \nu_n(F)\leq -\inf_{x\in F}I(x),
 \ee
 and for any open set $G$ in $\mathbb{R}$ we have  the lower estimate
  \be\label{ldp2}
 \liminf_{n\ra \infty}\frac{1}{n}\log \nu_n(G)\geq -\inf_{x\in G}I(x).
 \ee
\end{theorem}

\begin{remark}  The condition \rf{mgf2} holds when $W^L, W^R$ are Bernoulli random variables taking values $1$ and $-1$ with respective parameters $p_L, p_R$. The plots below show the rate function $I(x)$ in the interval $[-1,1] $ for some special values of $p_{\max}=p_L\vee p_R,p_{\min}=p_L\wedge p_R$ in this case.
\begin{figure}[H]
  \centering
    \includegraphics[width=2.5in]{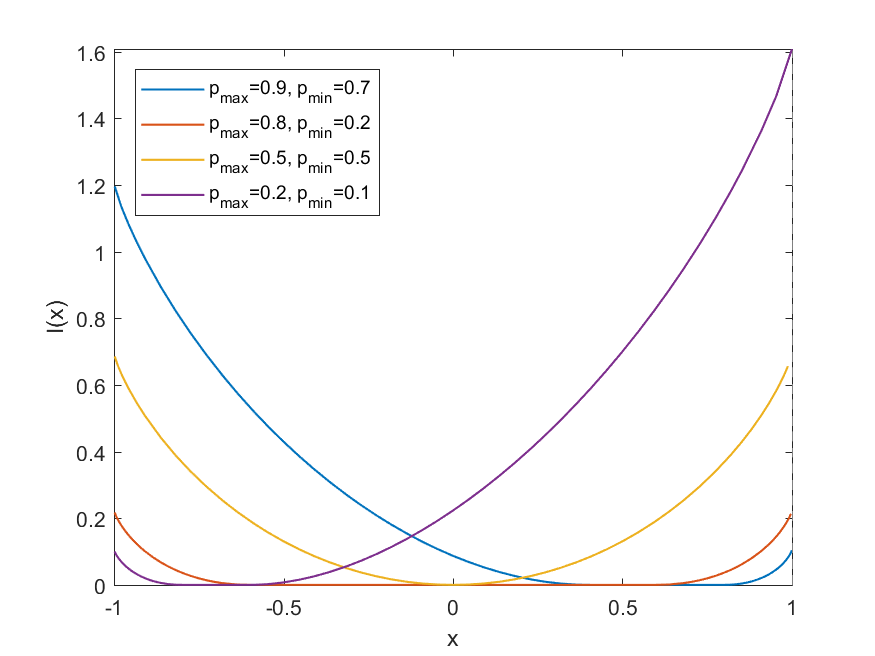}\\
  \caption{Rate Function}\label{rf-fig}
\end{figure}

A notable feature is the fact  that the rate function takes the value zero over an interval instead of a single point.

\end{remark}

\section{Applications}\label{application}
In this section, we discuss some applications of our strategic limit theorems including answers to the three questions about TAB problem mentioned  in Section~\ref{introduction}.

\subsection{Parameter Estimation}\label{para-estimation}  In general $\om$ and $\um $ defined in (\ref{musigma}) are unknown.  By the  law of large numbers \rf{lln-4}, we can find strategy $\theta^{\gamma}$  for any  given $\gamma$ such that
the unknown parameter  $h= \gamma\om+(1-\gamma)\um$ can be estimated by a sampling mean $S_n^{\theta^{\gamma}}/n$. This generalizes the law of large numbers obtained in \cite{robbins}, which corresponds to our result with $\gamma=0$ or $1$.

 An interesting scenario is as follows: Assume that one arm has a positive expected reward, and the other one has a negative expected reward. But it is not clear which one has the positive expected reward. One  would like to design a fair game so that
 the average reward is asymptotically $0$.  One answer, based on the law of large numbers \rf{lln-4}, is the strategy $\theta^{\gamma^*}$, where $\gamma^*=\frac{\um}{\um-\om}$.

 The only thing that remains is to determine the arm that has the maximum mean $\om$.

  \subsection{Parrondo's Paradox}\label{parrondo}
 Parrondo's paradox \cite{HA1999,HA1999-2,HA2002} was inspired by a class of physical systems: the Brownian ratchets \cite{Ajdari,Astumian,Linke,Magnasco,Reimann} and has received the attention of scientists from different fields, ranging from biology to economics. An important issue, addressed by various authors, is the reason why Parrondo's paradox holds. There are different explanations for the random mixture version and the non-random pattern version of the paradox.
 Hendrik Moraal \cite{Moraal} explained  that  this behaviour is due to two general features: (i) the class of games is such that mixing the playing of two games is equivalent to playing a third one, and (ii) the break-even boundaries for these games are curved. These features make it possible for researchers to determine when the Parrondo's paradox does not hold. Applying our law of large numbers \rf{lln-1}, we obtain the following corollary which confirms the role of "dependence" for the occurrence of Parrondo's paradox.

\begin{corol} In TAB problem, with the notations in Section~\ref{assump}, if the random rewards of the two arms are independent of "historic information"  individually, that is
$$
E_P[W_n^L|\CH^{\theta}_{n-1}]=E_P[W_n^L],\quad E_P[W_n^R|\CH^{\theta}_{n-1}]=E_P[W_n^R],\ n\ge1, \ \theta\in\Theta,
$$
then the Parrondo's paradox will not occur under any strategy.
\end{corol}
\begin{remark} The dependence between the random rewards of the two arms is also necessary for the occurrence the Parrondo's paradox. If the random rewards of the two arms are independent individually, then it follows from the law of large numbers \rf{lln-1} that  for any strategy $\theta$, the average reward $S_n^{\theta}/n$ will not exceed the maximal expected rewards of the two arms. Therefore the Parrondo's Paradox  does not hold.
\end{remark}
\subsection{Hypothesis Testing}\label{section-test}

In this section, we consider the hypothesis test for the TAB problem using Corollary~\ref{sym-max}.
More specifically,  we would like to determine  which arm provides the higher expected reward when $\om$ and $\um $ are known.  In other words we would like to conduct the (order) hypothesis test:
\begin{align}  \textbf{H}_0: (\mu_L,\mu_R)=(\om,\um)\ \text{ versus }\  \textbf{H}_1: (\mu_L,\mu_R)=(\um,\om).\tag{T1}
\end{align}
For the purpose of demonstration, we only consider the case that $\om=-\um$.  The general case holds similarly with minor adjustment.

\vspace{0.3mm}
For any $0<\alpha<1/2$,   let $z_{\alpha}$ be such that
  $$
 \lim_{n\ra \infty} P\lt(\lt|\hat{T}_{n,n}^{\hat{\theta}^{n,0}}\rt|> z_{\alpha}\rt)=\alpha,
 $$
 where the statistic $\hat{T}_{n,n}^{\hat{\theta}^{n,0}}$ and the strategies $\{\hat{\theta}^{n,0}:n\ge1\}$ are given in (\ref{Tmn}) and (\ref{strategy-hypo}). Equivalently
 \[
\Phi\left(\om+z_{\alpha} \right)  -e^{-2\om z_{\alpha}}\Phi\left( \om-z_{\alpha} \right) =1-\alpha.
 \]

 Since the strategy $\hat{\theta}^{n,0}$ is explicit,  by  (1) of Corollary~\ref{sym-max}, $\hat{T}_{n,n}^{\hat{\theta}^{n,0}}$ can serve as the test statistic for the above test. The occurrence of
$$ \lt|\hat{T}_{n,n}^{\hat{\theta}^{n,0}}\rt|> z_{\alpha}$$
for large enough $n$ will lead to the rejection of  ${\bf H}_0$ at the significance level $\alpha$.

By  (2) of Corollary~\ref{sym-max}, for a fixed large enough $n$, the related statistical power can be approximately calculated as
\begin{equation}\label{typeII-error-eq2}
1-\hat{\beta}=P\lt(\lt|\hat{T}_{n,n}^{\hat{\theta}^{n,0}}\rt|> z_{\alpha}\big|{\bf H}_1\rt)\approx1-\Phi\left(\frac{z_{\alpha}}{\hat{\sigma}}-\hat{\alpha}_n\right)  +e^{\frac{2\hat{\alpha}_n z_{\alpha}}{\hat{\sigma}}}\Phi\left( -\frac{z_{\alpha}}{\hat{\sigma}}-\hat{\alpha}_n\right).
\end{equation}

According to the traditional method of hypothesis testing, one usually uses the strategy $\theta=(1,1,1,\cdots)$ to obtain a sequence of data $\{Z_1^{\theta},Z_2^{\theta},\cdots\}$, that is, all the data are observed from a single arm. The test statistic is $$M_n:=\frac{1}{\sigma\sqrt{n}}\sum_{i=1}^n(Z_i^{\theta}-\om).$$
Given a significance level $\alpha>0$, the occurrence of $|M_n|>u_{\alpha/2}$, where $\Phi(u_{\alpha/2})=1-\alpha/2$, for large enough $n$ will lead to the rejection of   ${\bf H}_0$ at the significance level $\alpha$. For a fixed large enough $n$, the related statistical power can be approximately calculated as
\begin{equation}\label{typeII-error-eq}
1-\beta=P\lt(|M_n|> u_{\alpha/2}\,\big|{\bf H}_1\rt)\approx1-\Phi\lt(\frac{2\om}{\sigma}\sqrt{n}+u_{\alpha/2}\rt)+\Phi\lt(\frac{2\om}{\sigma}\sqrt{n}-u_{\alpha/2}\rt).
\end{equation}

At the end of this section,  to give a simulation of our hypothesis testing method, we consider a special case that the two arms with Bernoulli rewards,
\begin{equation}\label{q1q2}
\left\{
\begin{array}{l}
P(W^L=1)=p_L\\
P(W^L=-1)=1-p_L
\end{array}
\right.\text{ and }\quad
\left\{
\begin{array}{l}
P(W^R=1)=p_R\\
P(W^R=-1)=1-p_R
\end{array}
\right..
\end{equation}
where  $0<p_L,p_R<1$. Let $p_{\max}=\max\{p_L,p_R\}$ and $p_{\min}=\min\{p_L,p_R\}$, it is equivalent to consider the following hypothesis test
 \begin{align}  \textbf{H}_0: (p_L,p_R)=(p_{\max},p_{\min})\ \text{ versus }\  \textbf{H}_1: (p_L,p_R)=(p_{\min},p_{\max}).\tag{T2}
\end{align}
To keep the common variance, we also assume $p_L+p_R=1$, and then $$\sigma:=\os=\us=2\sqrt{p_{\max}p_{\min}}\ \text{ and also }\ \om=-\um.$$
The following figures indicate that the statistical power of our test method is larger than the statistical power under the traditional method. The significance level $\alpha$ is set at $0.05$.  The blue curve represents the statistical power of our test method, the red one represents the  statistical power of traditional method.
 \begin{figure}[H]
 \begin{minipage}{7cm}
  \centering
    \includegraphics[width=2.4in]{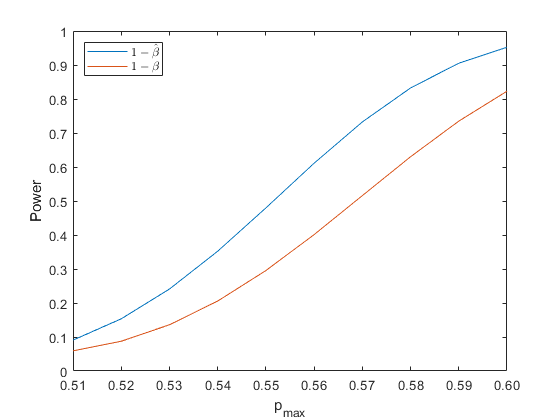}
    \caption{$p_{\max}$ values from $0.51$ to $0.60$, $n=50$}
  \end{minipage}
 \begin{minipage}{7cm}
  \centering
    \includegraphics[width=2.4in]{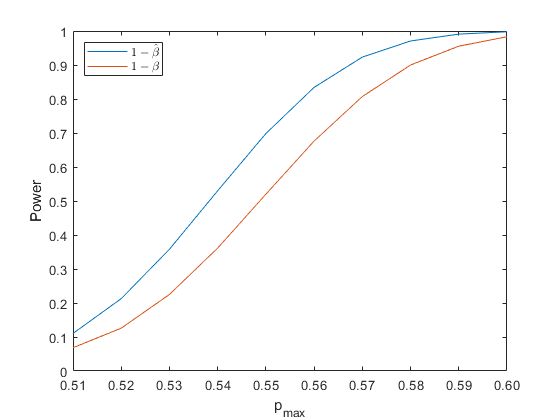}
    \caption{$p_{\max}$ values from $0.51$ to $0.60$, $n=100$}
     \end{minipage}
\end{figure}
 \begin{figure}[H]
 \begin{minipage}{7cm}
  \centering
    \includegraphics[width=2.4in]{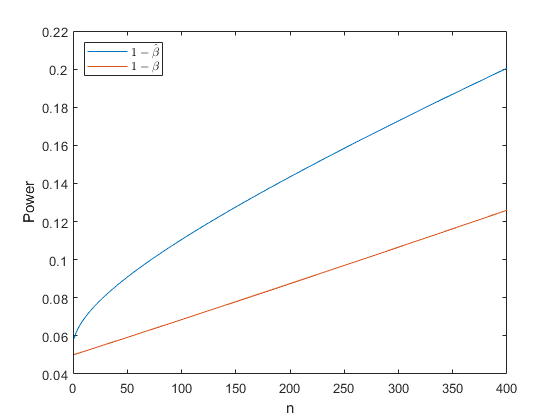}
    \caption{$n$ values from $1$ to $400$, $(p_{\max},p_{\min}) =(0.51,0.49)$}
  \end{minipage}
 \begin{minipage}{7cm}
  \centering
    \includegraphics[width=2.4in]{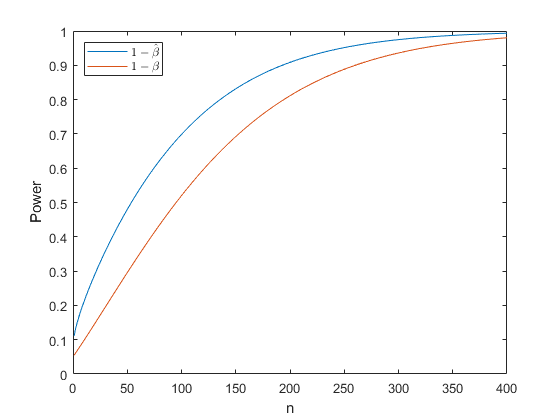}
    \caption{$n$ values from $1$ to $400$,  $(p_{\max},p_{\min})=(0.55,0.45)$}
  \end{minipage}
    \end{figure}
 \begin{figure}[H]
  \centering
    \includegraphics[width=2.4in]{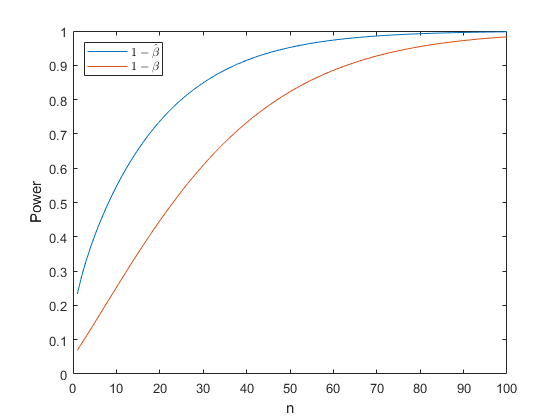}
    \caption{$n$ values from $1$ to $100$, $(p_{\max},p_{\min})=(0.6,0.4)$ }

\end{figure}
\section{Proofs}\label{proof}

\subsection{Proof of Lemma~\ref{tab-proper}}
\begin{proof}
 (1) For any
$\theta\in\Theta$ and $n\ge1$, it follows from the definitions of $\{W^L_i:i\ge1\}$ and $\{W^R_i:i\ge1\}$ that
\begin{align*}
E_P[Z_n^{\theta}|\CH^{\theta}_{n-1}]=&E_P[I_{\{\vartheta_n=1\}}W_n^L+I_{\{\vartheta_n=2\}}W_n^R|\CH^{\theta}_{n-1}]\\
=&I_{\{\vartheta_n=1\}}E_P[W_n^L]+I_{\{\vartheta_n=2\}}E_P[W_n^R]\\
=&I_{\{\vartheta_n=1\}}\mu_L+I_{\{\vartheta_n=2\}}\mu_R.
\end{align*}
Then we have
$$\esssup_{\theta\in\Theta}E_P[Z_n^{\theta}|\CH^{\theta}_{n-1}]=\om,\ \essinf_{\theta\in\Theta}E_P[Z_n^{\theta}|\CH^{\theta}_{n-1}]=\um.$$

Similarly, we can prove
\begin{align*}
\esssup\limits_{\theta\in\Theta}E_{P}\lt[\lt(Z_n^{\theta}-E_{ P}[Z_n^{\theta}|\CH^{\theta}_{n-1}]\rt)^2|\CH^{\theta}_{n-1}\rt]=\os^2,\\
\essinf\limits_{\theta\in\Theta}E_{P}\lt[\lt(Z_n^{\theta}-E_{ P}[Z_n^{\theta}|\CH^{\theta}_{n-1}]\rt)^2|\CH^{\theta}_{n-1}\rt]=\us^2.
\end{align*}
\vspace{0.2cm}

(2) For any $\theta\in\Theta$ and $n\geq 1$, let $U_{n-1}^{\theta}$ be a  $\CH^{\theta}_{n-1}$-measurable random variable, which is depend on $(\vartheta_1,\cdots,\vartheta_{n-1})$. By direct calculation we obtain that
\begin{align*}
&\sup\limits_{ \theta\in\Theta}E_{ P}\lt[\psi\lt(U_{n-1}^{\theta},Z_n^{\theta}\rt)\rt]\\
=&\sup\limits_{ \theta\in\Theta}E_{ P}\lt[I_{\{\vartheta_n=1\}}E_P[\psi\lt(U_{n-1}^{\theta},W_n^{L}\rt)|\CH^{\theta}_{n-1}]+I_{\{\vartheta_n=2\}}E_P[\psi\lt(U_{n-1}^{\theta},W_n^{R}\rt)|\CH^{\theta}_{n-1}]\rt]\\
=&\sup\limits_{ \theta\in\Theta}E_{P}\lt[\psi_n^L\big(U_{n-1}^{\theta}\big)\vee\psi_n^R\big(U_{n-1}^{\theta}\big)\rt],
\end{align*}\normalsize
where for $ x\in\CR$
\beq
\psi^L_n(x)&=&E_{P}[\psi(x,W^L_n)]=f_0(x)+\mu_L\, f_1(x)+(\mu_L^2+\sigma_L^2)\, f_2(x),\\
\psi^R_n(x)
&=&E_{P}[\psi(x,W^R_n)]=f_0(x)+\mu_R\, f_1(x)+(\mu_R^2+\sigma_R^2)\, f_2(x).
\eeq
\end{proof}

\subsection{Proof of Law of Large Numbers}\label{seclln}

\begin{proof}[ Proof of Theorem~\ref{thm-lln}]
Proof of  \rf{lln-4}:
	Firstly, for any $\gamma\in[0,1]$, we construct the strategy $\theta^{\gamma}=(\vartheta_1^{\gamma},\cdots,\vartheta_n^{\gamma},\cdots)$ as following,
\begin{description}
   \item[Round $1$:] Choose arm {\bf L}, that is $\vartheta_1^{\gamma}=1$.
   \item[Round $2$:] Choose arm {\bf R}, that is $\vartheta_2^{\gamma}=2$.
   \item[Round $k(\ge3)$:]  We discuss the $k$ in the following cases.

 \begin{itemize}
    \item  Case $\mathcal{A}$: When $k=2^i-1$, for some $i>1$. We choose arm {\bf L}, that is $\vartheta_k^{\gamma}=1$.
	\item  Case $\mathcal{B}$: When $k=2^i$, for some $i>1$. We choose arm {\bf R}, that is $\vartheta_k^{\gamma}=2$.
	
	\item  Case $\mathcal{C}$: When $2^i<k<2^{i+1}-1$, for some $i>1$. Let $m_{k}^{L}$ (respectively, $m_{k}^{R}$) be the number of times that arm {\bf L} (respectively, {\bf R}) are chose in previous   $k$ rounds, that is, the number of times that $1$ (respectively, $2$) appears among $\{\vartheta_1^{\gamma}, \ldots, \vartheta_{k}^{\gamma}\}$. Let $\mu_k^L$ and $\mu_k^R$ be the arithmetic mean of all previous observation of arm {\bf L} and arm {\bf R} up to stage $k$ respectively, that is,
	 $$\mu_k^L:=\frac{\sum_{j\le k,\vartheta_j^{\gamma}=1}W_{j}^L}{m_k^L}, \quad \mu_k^R:=\frac{\sum_{j\le k,\vartheta_j^{\gamma}=2}W_{j}^R}{m_k^{R}}.$$	

	 Now $\mathcal{C}$ can be divided into two sub-cases.
	
	 Case $\mathcal{C}.1$: When  $\mu_{k-1}^L\geq\mu_{k-1}^R$. We choose arm {\bf L} if $\frac{m_{k-1}^L}{k-1}<\gamma$, and choose arm {\bf R} if  $\frac{m_{k-1}^L}{k-1}\ge\gamma$.
	
	 Case $\mathcal{C}.2$: When $\mu_{k-1}^L<\mu_{k-1}^R$. We choose arm {\bf R} if $\frac{m_{k-1}^R}{k-1}<\gamma$, and choose arm {\bf L} if  $\frac{m_{k-1}^R}{k-1}\ge\gamma$.
 \end{itemize}

 \end{description}

	Now we will show that the convergence holds under the strategy $\theta^{\gamma}$.

	 By the strong law of large numbers, we know that
	 $$\mu_k^L\to\mu_L, \mu_k^R\to\mu_R,\ P\text{-a.s.},\text{ as }k \to \infty.$$
	
	 Thus there exist a set $\Omega_0\in\mathcal{F}$, such that $P(\Omega_0)=1$, and for any $\omega\in\Omega_0,$  we have
	 $$\mu_k^L(\omega)\to\mu_L, \mu_k^R(\omega)\to\mu_R,\text{ as }k \to \infty.$$
	
	 Given a small enough $\epsilon<\frac{|\mu_L-\mu_R|}{2}$ (with the assumption that $\mu_L\neq\mu_R$), we know that there exists $N(\omega)>0$, such that
	 $$|\mu_k^L(\omega)-\mu_L|\leq\epsilon,|\mu_k^R(\omega)-\mu_R|\leq\epsilon,\text{ for any }k>N(\omega).$$
This implies that either $\mu_{k}^L(\w)>\mu_k^{R}(\w)$ for all $k \geq N(\omega)$ or $\mu_{k}^L(\w)<\mu_k^{R}(\w)$ for all $k \geq N(\omega)$.
In other words when $k$ is large enough, we will always choose arm {\bf L} as long as $\frac{m_{k-1}^L}{k-1}<\gamma$, or always choose arm {\bf R} as long as $\frac{m_{k-1}^R}{k-1}<\gamma$.
	 Thus we have
	 \begin{align*}
	 	\frac{S_n^{\theta^{\gamma}}(\omega)}{n}	 	
=&\frac{\sum_{k\le n,\vartheta_k^{\gamma}=1}W_k^L(\omega)}{n}
	 	+\frac{\sum_{k\le n,\vartheta_k^{\gamma}=2}W_k^R(\omega)}{n}\\
	 =&I_{\lt[\bigcap_{k=N(\w)}^{\infty}\{\mu_{k}^L(\w)>\mu_k^{R}(\w)\}\rt]}\lt(\frac{m_n^{L}}{n}\mu_n^{L}
	 	+\frac{n-m_n^{L}}{n}\mu_n^{R}\rt)\\
&+I_{\lt[\bigcap_{k=N(\w)}^{\infty}\{\mu_{k}^L(\w)<\mu_k^{R}(\w)\}\rt]}\lt(\frac{n-m_n^{R}}{n}\mu_n^{L}
	 	+\frac{m_n^{R}}{n}\mu_n^{R}\rt)\\
\to&I_{\lt[\bigcap_{k=N(\w)}^{\infty}\{\mu_{k}^L(\w)>\mu_k^{R}(\w)\}\rt]}\lt(\gamma\mu_L+(1-\gamma)\mu_R\rt)\\
&+I_{\lt[\bigcap_{k=N(\w)}^{\infty}\{\mu_{k}^L(\w)<\mu_k^{R}(\w)\}\rt]}\lt((1-\gamma)\mu_L+\gamma\mu_R\rt)\\
=&\gamma\max\{\mu_L,\mu_R\}+(1-\gamma)\min\{\mu_L,\mu_R\}=h,\quad \text{as }n\to\infty.
	 \end{align*}

\vspace{0.4cm}
Proof of  \rf{lln-1}:  For any $\ep >0$, it suffices to prove
 \begin{equation}\label{lln-prove-1}
\lim\limits_{n\rightarrow\infty}\sup_{\theta\in\Theta}P\lt(\frac{S_n^{\theta}}{n}\le\um-\ep \rt)=0\ \text{ and }\lim\limits_{n\rightarrow\infty}\sup_{\theta\in\Theta}P\lt(\frac{S_n^{\theta}}{n}\ge\om+\ep \rt)=0.
\end{equation}
Now we give the proof of the first equation in (\ref{lln-prove-1}), the other one can be proved similarly.

For any integer $m\geq 1$, let $C_b^m(\CR)$ denote the set of functions on $\CR$ that have bounded derivatives up to order $m$. Let $\phi\in C_b^2(\CR)$ be an decreasing function such that $I_{\{x\le \um-\ep\}}\le \phi(x)$ and $\phi(\um)=0$. Then, applying the Taylor's expansion, we have
\begin{align*}
&\sup_{\theta\in\Theta}P\lt(\frac{S_n^{\theta}}{n}\le\um-\ep \rt)\\
\le&\sup_{\theta\in\Theta}E_P\lt[\phi\lt(\frac{S_n^{\theta}}{n} \rt)\rt]-\phi(\um)\\
=&\sum_{m=1}^n\lt\{\sup_{\theta\in\Theta}E_P\lt[\phi\lt(\frac{S_m^{\theta}}{n}+\frac{n-m}{n}\um \rt)\rt]-\sup_{\theta\in\Theta}E_P\lt[\phi\lt(\frac{S_{m-1}^{\theta}}{n}+\frac{n-m+1}{n}\um \rt)\rt]\rt\}\\
\le&\sum_{m=1}^n\sup_{\theta\in\Theta}E_P\lt[\dot{\phi}\lt(\frac{S_{m-1}^{\theta}}{n}+\frac{n-m+1}{n}\um \rt)\frac{Z_m^{\theta}-\um}{n}\rt]+C_0\sum_{m=1}^n\sup_{\theta\in\Theta}E_P\lt[\frac{(Z_m^{\theta}-\um)^2}{n^2}\rt]\\
\le&\frac{C_0(\os^2+(\om-\um)^2)}{n}\to0,\quad \text{ as }n\to\infty,
\end{align*}
where  the number of dots on top of a function denote the same order derivatives with respect to $x$, $C_0=\sup_{x\in\CR}|\ddot{\phi}(x)|$ is the bound of $\ddot{\phi}$, and the convergence is due to the finiteness of $\os,\um$ and $\om$. Then we complete the proof of (\ref{lln-prove-1}).

%
%
%
%
%

\vspace{0.4cm}

 Proof of  \rf{lln-2}: For any $h\in [\um,\om]$, there exists  $0\leq \gamma\leq 1$ such that $h=\gamma \om+(1-\gamma)\um$. By (\ref{lln-4}), we have  that for any $\ep >0$
 \[
1\geq \sup_{\theta\in\Theta}P\bigg( \bigg|\frac{S_n^{\theta}}{n}-h\bigg|<\ep\bigg)\geq  P\bigg( \bigg|\frac{S_n^{\theta^{\gamma}}}{n}-h\bigg|<\ep\bigg)  \ra 1,\ \mbox{as}\  n \ra \infty.
\]
\end{proof}

\subsection{Proof of Strategic  Central Limit Theorem}\label{secclt}
The main idea of the proof is based on piecewise comparison between our approximating sequences and the solution of a stochastic differential equation (SDE). In comparison with the methods in earlier work \cite{peng2019-1,peng2019,chenepstein}, where the limit was derived from a sequence of maximal expectations  and was identified through solutions of partial differential equations (PDE) and a class of backward stochastic differential equations (BSDEs),  our method is based on a direct and explicit construction of the optimal strategy, which helps identify the limit and avoids the use of nonlinear  BSDEs and PDE.


We begin with a discussion of a the SDE and thus the limit distribution. This is followed by a few technical lemmas. The proofs of strategic central limit theorems will be presented afterwards.

Let $\{B_{s}\}_{s\ge0}$ be the standard Brownian motion  on  $(\Omega,\mathcal{F},P)$ and
 $(\mathcal{F}^*_{s})_{s\ge0}$ be the natural
filtration generated by $\{B_{s}\}_{s\ge0}$.

For any fixed $c\in\CR$ and any $(t,x,\alpha)\in[0,1]\times\CR\times\CR$, let $\{Y_s^{t,x,\alpha,c}\}_{s\in[t,1]}$ denote  the solution of the SDE
\begin{equation}\label{sde}
\left\{
\begin{array}{l}
dY_s^{t,x,\alpha,c}=\alpha sgn\lt(Y_s^{t,x,\alpha,c}-c\rt)ds+dB_s,\quad s\in[t,1]\\
Y_t^{t,x,\alpha,c}=x.
\end{array}
\rt.
\end{equation}
Although the drift coefficient is discontinuous, this equation does have a unique strong solution (see \cite[Theorem 1]{mel}).
For a general reference on  SDEs with two-valued drift, we refer to \cite{Graversen} and \cite{KS1984}.

{The following lemma  is essentially Proposition~5.1 in \cite{KS1984}, which shows the the connection between  $f^{\alpha,\beta,c}$ given in (\ref{proba-density}) and the probability density of $\{Y_s^{t,x,\alpha,c}\}_{s\in [t,1]}$.

\begin{lemma}\label{clt-l1}
The transition probability density of the process  $\{Y_s^{t,x,\alpha,c}\}_{s\in [t,1]}$ is given by
\beq
q_{\alpha,c}(t,x;s,z)=&\frac{1}{\sqrt{2\pi (s-t)}}e^{-\frac{(x-z)^2-2\alpha (s-t)(|z-c|-|x-c|)+\alpha^2(s-t)^2}{2(s-t)}}\nonumber\\
&-\alpha e^{2\alpha |z-c|}\int_{|x-c|+|z-c|+\alpha (s-t)}^{\infty}\frac{1}{\sqrt{2\pi (s-t)}}e^{-\frac{u^2}{2(s-t)}}du,\quad  
\eeq
for any $0\leq t<s\leq1$ and $z\in\CR$.

Particularly, when $t=0$ and $x=0$, we have
$$q_{\alpha,c}(0,0;1,z)=f^{\alpha,0,c}(z),$$
where $f^{\alpha,0,c}(z)$ is the probability density  given in (\ref{proba-density}).
\end{lemma}}

For  any $\ph\in C_b^3(\CR)$ that is symmetric with centre $c$ and any $t$ in [0,1],
we define
\begin{align}
H_t(x)=E_{P}\lt[\ph\lt(Y_1^{t,x,\alpha,c}\rt)\rt],\quad x\in\CR,\label{hf}
\end{align}
where the dependence on $\ph$, $\alpha$ and $c$  is not explicitly noted for simplicity.

It is clear from the definition that

$$H_1(x)=\ph(x),\quad H_0(0)=E_{P}[\ph(Y_1^{0,0,\alpha,c})]=\int_{\CR}\ph(y)f^{\alpha,0,c}(y)dy.
$$
where $f^{\alpha,0,c}(y)$ is given in (\ref{proba-density}).

The following lemma lists some analytic properties of the family $\{H_t(x)\}_{t\in[0,1]}$.

\begin{lemma}
\label{lemma-ddp} Let the number of dots on top of a function denote the same order derivatives with respect to $x$.

\begin{description}

\item [(1)] For each fixed $t\in[0,1]$, $H_{t}(x) \in C_b^{2}(\CR)$. In addition,
     the first and second order derivatives of $H_{t}(x)$  are uniformly bounded  for all $0\leq t \leq 1$ and $x$.

 \item [(2)] The family $\{\ddot{H}_t(x)\}_{t\in[0,1]}$ is uniformly Lipschitz, i.e., there exists a constant $L$, independent with $t$,  such that
    $$\lt| \ddot{H}_t(x_1)-\ddot{H}_t(x_2)\rt|\leq L|x_1-x_2|, \ \ x_1,x_2\in\CR.
    $$

\item [(3)] For any $t\in[0,1]$, $H_t(x)$  is symmetric with centre $c$. Furthermore, if
 for any $x\in\CR,$
 $$sgn(\dot{\ph}(x))=\pm sgn(x-c),$$
  then
  $$sgn (\dot{H}_{t}(x))=\pm sgn(x-c),\  x\in\CR.$$

\item [(4)] Markov property: for any $t\in[0,1)$ and $h\in[0,1-t]$, $$H_{t}(x)=E_{P}\lt[H_{t+h}\lt(Y_{t+h}^{t,x,\alpha,c}\rt)\rt],\ x\in \CR.
    $$

\item [(5)] If $sgn(\dot{\ph}(x))=\pm sgn(x-c)$ for all $x\in\CR$, then $$\lim_{n\to\infty }\sum_{m=1}^{n}\sup\limits_{x\in\mathbb{R}}\left\vert H_{\frac{m-1}{n}}\left(  x\right)
-H_{\frac{m}{n}}\left(  x\right)  \mp \frac{\alpha}{n}\left|  \dot{H}_{\frac{m}{n}}%
(x)\right|  -\frac{1}{2n}\ddot{H}_{\frac{m}{n}}(x)\right\vert =0.$$

\end{description}
\end{lemma}

\begin{proof} We prove the lemma in numerical order.

(1)  For $t=1$, $H_1(x)\equiv\ph(x)$ and the result is trivial.

Next we assume that  $0\leq t <1$.  By Lemma~\ref{clt-l1}, we have

\begin{align*}
H_{t}(x)=\int_{-\infty}^{\infty}\ph(z)q_{\alpha,c}(t,x;1,z)dz,\quad \forall x\in\CR.
\end{align*}
Since $\ph$ is symmetric with centre $c$, we obtain
$$H_t(x)=\int_{0}^{\infty}\ph(z+c)\big(q_{\alpha,c}(t,x;1,z+c)+q_{\alpha,c}(t,x;1,-z+c)\big)dz.$$
It follows by direct calculation that
\beqn
\dot{H}_{t}(x)&=&\int_{0}^{\infty}\frac{sgn(x-c)}{\sqrt{2\pi (1-t)}}\dot{\ph}(z+c)e^{-\tfrac{(z-\alpha (1-t)-|x-c|)^2}{2(1-t)}}\lt[1-e^{-\tfrac{2|x-c|z}{1-t}}\rt]dz,\label{firstde}\\
\ddot{H}_{t}(x)&=&\int_{0}^{\infty}\frac{1}{\sqrt{2\pi (1-t)}}\ddot{\ph}(z+c)e^{-\tfrac{(z-\alpha (1-t)-|x-c|)^2}{2(1-t)}}\lt[1+e^{-\tfrac{2|x-c|z}{1-t}}\rt]dz\nonumber\\
&&+\int_{0}^{\infty}\frac{2\alpha}{\sqrt{2\pi (1-t)}}\dot{\ph}(z+c)e^{-\tfrac{(z+\alpha (1-t)+|x-c|)^2}{2(1-t)}}e^{2\alpha z}dz\nn\label{secondde}\\
&=&\int_{0}^{\infty}\frac{1}{\sqrt{2\pi (1-t)}}\ddot{\ph}(z+c)e^{-\tfrac{(z-\alpha (1-t)-|x-c|)^2}{2(1-t)}}\lt[1+e^{-\tfrac{2|x-c|z}{1-t}}\rt]dz\nonumber\\
&&+\int_{0}^{\infty}\frac{2\alpha}{\sqrt{2\pi (1-t)}}\dot{\ph}(z+c)e^{-\tfrac{(z-\alpha (1-t)+|x-c|)^2}{2(1-t)}}e^{-2\alpha |x-c|}dz.\nn
\eeqn
Since $\ph\in C_b^3(\CR)$, we conclude that $H_t(x)\in C_b^2(\CR)$, and the first and second order derivatives of $H_t(x)$ are uniformly bounded for all $t$ and $x$.

\vspace{0.4cm}

(2) For $x<c$, we have
\begin{align*}
\dddot{H}_{t}(x)=&\int_{0}^{\infty}\frac{1}{\sqrt{2\pi (1-t)}}\dddot{\ph}(z+c)e^{-\tfrac{(z-\alpha (1-t)+x-c)^2}{2(1-t)}}\lt[e^{\tfrac{2(x-c)z}{1-t}}-1\rt]dz\\
&+\int_{0}^{\infty}\frac{4\alpha}{\sqrt{2\pi (1-t)}}\lt[\alpha \dot{\ph}(z+c)+\ddot{\ph}(z+c)\rt]e^{-\tfrac{(z+\alpha (1-t)-x+c)^2}{2(1-t)}}e^{2\alpha z}dz\\
=&\int_{0}^{\infty}\frac{1}{\sqrt{2\pi (1-t)}}\dddot{\ph}(z+c)e^{-\tfrac{(z-\alpha (1-t)+x-c)^2}{2(1-t)}}\lt[e^{\tfrac{2(x-c)z}{1-t}}-1\rt]dz\\
&+\int_{0}^{\infty}\frac{4\alpha}{\sqrt{2\pi (1-t)}}\lt[\alpha \dot{\ph}(z+c)+\ddot{\ph}(z+c)\rt]e^{-\tfrac{(z-\alpha (1-t)-x+c)^2}{2(1-t)}}e^{2\alpha (x-c)}dz.
\end{align*}

For $x>c$, we have
\begin{align*}
\dddot{H}_{t}(x)=&\int_{0}^{\infty}\frac{1}{\sqrt{2\pi (1-t)}}\dddot{\ph}(z+c)e^{-\tfrac{(z-\alpha (1-t)-x+c)^2}{2(1-t)}}\lt[1-e^{-\tfrac{2(x-c)z}{1-t}}\rt]dz\\
&-\int_{0}^{\infty}\frac{4\alpha}{\sqrt{2\pi (1-t)}}\lt[\ddot{\ph}(z+c)+\alpha \dot{\ph}(z+c)\rt]e^{-\tfrac{(z+\alpha (1-t)+x-c)^2}{2(1-t)}}e^{2\alpha z}dz\\
=&\int_{0}^{\infty}\frac{1}{\sqrt{2\pi (1-t)}}\dddot{\ph}(z+c)e^{-\tfrac{(z-\alpha (1-t)-x+c)^2}{2(1-t)}}\lt[1-e^{-\tfrac{2(x-c)z}{1-t}}\rt]dz\\
&-\int_{0}^{\infty}\frac{4\alpha}{\sqrt{2\pi (1-t)}}\lt[\ddot{\ph}(z+c)+\alpha \dot{\ph}(z+c)\rt]e^{-\tfrac{(z-\alpha (1-t)+x-c)^2}{2(1-t)}}e^{-2\alpha (x-c)}dz.
\end{align*}
Since $\ph\in C_b^3(\CR)$, it follows that $\dddot{H}_t(x)$ is uniformly bounded for all $t$ and $x\neq c$. For $x=c$, the third order left and right derivatives of $H_t(x)$ can be shown to exist and are also bounded uniformly in $t$. Thus by the mean value theorem one can find a constant $L$,  independent with $t$, such that
 for any $x_1,x_2\in\CR$,
$$\lt|\ddot{H}_{t}(x_1)-\ddot{H}_{t}(x_2)\rt|\leq L|x_1-x_2|.
$$

\vspace{0.4cm}

(3) It follows by direct calculation that for any $x,z\in\CR$ and $t\in[0,1)$,
$$q_{\alpha,c}(t,x+c;1,z+c)=q_{\alpha,c}(t,-x+c;1,-z+c).$$
Thus
\begin{align*}
H_t(x+c)=&\int_{0}^{\infty}\ph(z+c)\big(q_{\alpha,c}(t,x+c;1,z+c)+q_{\alpha,c}(t,x+c;1,-z+c)\big)dz\\
=&\int_{0}^{\infty}\ph(z+c)\big(q_{\alpha,c}(t,-x+c;1,-z+c)+q_{\alpha,c}(t,-x+c;1,z+c)\big)dz\\
=&H_t(-x+c)
\end{align*}
and  $H_t$ is symmetric with centre $c$.

By \rf{firstde}  we have that for any $x\in\CR$,
 $$sgn(\dot{H}_{t}(x))=\pm sgn(x-c) \text{ when } \ sgn(\dot{\ph}(x))=\pm sgn(x-c) .$$

\vspace{0.4cm}

(4) This follows from the Markov property of  $\{Y_s^{t,x,\alpha,c}\}_{s\in [t,1]}$, namely,
\begin{align*}
H_{t}(x)=E_{P}[\ph(Y_{1}^{t,x,\alpha,c})]=E_{P}\lt[E_{P}[\ph(Y_{1}^{t,x,\alpha,c})|\CF_{t+h}^*]\rt]=E_{P}[H_{t+h}(Y_{t+h}^{t,x,\alpha,c})].
\end{align*}

\vspace{0.4cm}

(5) We only prove the case  $sgn(\dot{\ph}(x))= sgn(x-c)$. The other case follows by similar arguments.  Applying the Markov property in   (4), we have for any $1\leq m\leq n$,\begin{align*}
H_{\frac{m-1}{n}}(x)=E_{P}\lt[H_{\frac{m}{n}}\lt(Y_{\frac{m}{n}}^{\frac{m-1}{n},x,\alpha,c}\rt)\rt].
\end{align*}
By It\^o's formula, we have
\begin{align*}
H_{\frac{m}{n}}\lt(Y_{\frac{m}{n}}^{\frac{m-1}{n},x,\alpha,c}\rt)=H_{\frac{m}{n}}\lt(x\rt)&+\int_{\frac{m-1}{n}}^{\frac{m}{n}}\dot{H}_{\frac{m}{n}}\lt(Y_s^{\frac{m-1}{n},x,\alpha,c}\rt)dY_s^{\frac{m-1}{n},x,\alpha,c}\\
&+\frac{1}{2}\int_{\frac{m-1}{n}}^{\frac{m}{n}}\ddot{H}_{\frac{m}{n}}\lt(Y_s^{\frac{m-1}{n},x,\alpha,c}\rt)ds.
\end{align*}
This combined with  (3) implies that
\begin{align*}
&\quad\ H_{\frac{m-1}{n}}\lt(x\rt)\\
&=E_{P}\lt[H_{\frac{m}{n}}\lt(x\rt)+\int_{\frac{m-1}{n}}^{\frac{m}{n}}\dot{H}_{\frac{m}{n}}\lt(Y_s^{\frac{m-1}{n},x,\alpha,c}\rt)dY_s^{\frac{m-1}{n},x,\alpha,c}+\frac{1}{2}\int_{\frac{m-1}{n}}^{\frac{m}{n}}\ddot{H}_{\frac{m}{n}}\lt(Y_s^{\frac{m-1}{n},x,\alpha,c}\rt)ds\rt]\\
&=E_{P}\bigg[H_{\frac{m}{n}}\lt(x\rt)+\int_{\frac{m-1}{n}}^{\frac{m}{n}}\alpha \dot{H}_{\frac{m}{n}}\lt(Y_s^{\frac{m-1}{n},x,\alpha,c}\rt)sgn\lt(Y_s^{\frac{m-1}{n},x,\alpha,c}-c\rt)ds\\
& \hspace{8.1cm}+\frac{1}{2}\int_{\frac{m-1}{n}}^{\frac{m}{n}}\ddot{H}_{\frac{m}{n}}\lt(Y_s^{\frac{m-1}{n},x,\alpha,c}\rt)ds\bigg]\\
&=E_{P}\lt[H_{\frac{m}{n}}\lt(x\rt)+\int_{\frac{m-1}{n}}^{\frac{m}{n}}\alpha\lt| \dot{H}_{\frac{m}{n}}\lt(Y_s^{\frac{m-1}{n},x,\alpha,c}\rt)\rt|ds+\frac{1}{2}\int_{\frac{m-1}{n}}^{\frac{m}{n}}\ddot{H}_{\frac{m}{n}}\lt(Y_s^{\frac{m-1}{n},x,\alpha,c}\rt)ds\rt].
\end{align*}

Taking the supremum over $x$, we obtain
\begin{align*}
&\sum_{m=1}^{n}\sup\limits_{x\in\mathbb{R}}\left\vert H_{\frac{m-1}{n}}\left(  x\right)
-H_{\frac{m}{n}}\left(  x\right)  -\frac{\alpha}{n}\left|  \dot{H}_{\frac{m}{n}}%
(x)\right|  -\frac{1}{2n}\ddot{H}_{\frac{m}{n}}(x)\right\vert\\
 \leq&\sum_{m=1}^{n}\sup\limits_{x\in\mathbb{R}}E_{P}\bigg[\int_{\frac{m-1}{n}}^{\frac{m}{n}}|\alpha|\lt| \dot{H}_{\frac{m}{n}}\lt(Y_s^{\frac{m-1}{n},x,\alpha,c}\rt)-\dot{H}_{\frac{m}{n}}(x)\rt|ds\\
& \hspace{3cm}+\frac{1}{2}\int_{\frac{m-1}{n}}^{\frac{m}{n}}\lt| \ddot{H}_{\frac{m}{n}}\lt(Y_s^{\frac{m-1}{n},x,\alpha,c}\rt)-\ddot{H}_{\frac{m}{n}}(x)\rt|ds\bigg]\\
\leq&\sum_{m=1}^{n}\sup\limits_{x\in\mathbb{R}}\frac{C}{n}E_{P}\lt[\sup_{s\in[\frac{m-1}{n},\frac{m}{n}]}\lt|Y_s^{\frac{m-1}{n},x,\alpha,c}-x\rt|\rt]\\
\leq&\sum_{m=1}^{n}\frac{C}{n}E_{P}\lt[\frac{|\alpha|}{n}+\sup_{s\in[\frac{m-1}{n},\frac{m}{n}]}|B_s-B_{\frac{m-1}{n}}|\rt]\\
\leq& C\lt(\frac{|\alpha|}{n}+\frac{1}{\sqrt{n}}\rt),
\end{align*}
where $C$ is a constant depending only on $\alpha, L$ and the bound of $\ddot{H}_t(x)$. This concludes the proof of the lemma.
\end{proof}

For both theorems, we prove them first for the special case where \begin{equation}\label{means-condition}
\om=-\um\ge0.
\end{equation}
Then the results asserted for general $\um$ and $\om$ are established by applying the preceding special case to $\{Y_i^{\theta}:i\ge1\}$, where $Y_i^{\theta}=Z_i^{\theta}-\frac{\om+\um}{2}$ and thus
$$
\esssup_{\theta\in\Theta}E_P[Y_i^{\theta}|\CH^{\theta}_{i-1}]=\frac{\om-\um}{2},\quad
\essinf_{\theta\in\Theta}E_P[Y_i^{\theta}|\CH^{\theta}_{i-1}]=-\frac{\om-\um}{2}.
$$
All results below are under the assumptions Theorem~\ref{thm-CLT0}.

The next lemma gives two remainder estimations that will be used repeatedly in the sequel.
 \begin{lemma}\label{remainder-lemma} Let  $\ph\in C_b^3(\CR)$ be symmetric with centre $c\in\CR$, and $\{H_t(x)\}_{t\in[0,1]}$ be defined as in (\ref{hf}). For any $\theta\in\Theta$, $n\in\mathbb{N}^+$ and $1\le m\le n$, set
\begin{align}
\Gamma(m,n,\theta)=& H_{\frac{m}{n}}(T_{m-1,n}^{\theta} )+ \dot{H}_{\frac{m}{n}}(T_{m-1,n}^{\theta} )\left(\frac {Z_m^{\theta}}{n} +\frac{\overline{Z}_m^{\theta}}{\sqrt{n}}\right)+\frac{1}{2} \ddot{H}_{\frac{m}{n}}(T_{m-1,n}^{\theta}) \left(\frac{\overline{Z}_m^{\theta}}{\sqrt{n}}\right)^2,\label{F}
 \end{align}
 where $\overline{Z}_m^{\theta}=\frac{1}{\sigma}(Z_m^{\theta}-E_P[Z_m^{\theta}|\CH^{\theta}_{m-1}])$.
Then
\begin{equation}\label{remainder-esti}
\lim_{n\to\infty}\sum_{m=1}^n\sup_{ \theta\in\Theta}E_{P}\left[  \left|H_{\frac{m}{n}}\left( T_{m,n}^{\theta}\right)- \Gamma(m,n,\theta)\right|\rt] =0.
\end{equation}
Furthermore, define a family of functions $\{L_{m,n}(x)\}_{m=1}^n$  by
\be\label{function-Lt}
L_{m,n}(x)=H_{\frac{m}{n}}(x)+\frac{\om}{n}\lt| \dot{H}_{\frac{m}{n}}(x)\rt|+\frac{1}{2n} \ddot{H}_{\frac{m}{n}}(x),\quad x\in\CR,
\ee
then we have
\begin{equation}\label{lemma-taylor1}
 \lim_{n\to\infty}\sum_{m=1}^n\left|\sup
\limits_{ \theta\in\Theta}E_{P}\left[  H_{\frac{m}{n}}\left( T_{m,n}^{\theta}\right)\right] - \sup
\limits_{ \theta\in\Theta}E_{P}\left[L_{m,n} \lt(T_{m-1,n}^{\theta} \rt)\right]\right|=0 .
\end{equation}\normalsize
\end{lemma}

\begin{proof}  In fact, by  (1) and  (2) of  Lemma~\ref{lemma-ddp},  there exists a constant $C>0$ such that
$$\sup\limits_{t\in[0,1]}\sup\limits_{x\in\CR} |\ddot{H}_{t}(x)|\leq C
   ,\quad\sup\limits_{t\in[0,1]}\sup\limits_{x,y\in\CR,x\neq y} \frac{|\ddot{H}_{t}(x)-\ddot{H}_t(y)|}{|x-y|}\leq C.$$

It follows from Taylor's expansion that for any $\ep>0$, there exists $\delta>0$ (depends only on $C$ and $\ep$), such that for any $x,y\in \CR$, and $t\in[0,1]$,
\begin{equation}\label{le0}
\left|H_{t}(x+y)-H_{t}(x)- \dot{H}_{t}(x)y- \frac 12  \ddot{H}_{t}(x)y^2\right| \leq \ep|y|^2I_{\{|y|<\delta\}}+C|y|^2I_{\{|y|\ge\delta\}}.
\end{equation}
For any   $1\le m \le n,$  taking $x=T_{m-1,n}^{\theta},y=\frac{Z_m^{\theta}}{n}+\frac{\overline{Z}_m^{\theta}}{\sqrt{n}}$ in  \rf{le0}, we obtain
\begin{align*}
&\sum_{m=1}^n\sup\limits_{ \theta\in\Theta} E_P\left[\left| H_{\frac{m}{n}}\left(T_{m,n}^{\theta} \right)-\Gamma( m ,n,\theta) \right|\rt]\\
\leq& \frac{C}{2}\sum_{m=1}^n \sup\limits_{ \theta\in\Theta}E_P\left[\lt|\frac {Z_m^{\theta}}{n}\rt|^2 +2\lt|\frac {Z_m^{\theta}}{n}\rt|\lt|\frac{\overline{Z}_m^{\theta}}{\sqrt{n}}\rt|\right]\\
&+\ep\sum_{m=1}^n\sup\limits_{ \theta\in\Theta}E_P\left[\left|\frac {Z_m^{\theta}}{n} +\frac{\overline{Z}_m^{\theta}}{\sqrt{n}}\right|^2I_{\{|\frac {Z_m^{\theta}}{n} +\frac{\overline{Z}_m^{\theta}}{\sqrt{n}}|<\delta\}}\rt]\\
&+C\sum_{m=1}^n\sup\limits_{ \theta\in\Theta}E_P\left[\left|\frac {Z_m^{\theta}}{n} +\frac{\overline{Z}_m^{\theta}}{\sqrt{n}}\right|^2I_{\{|\frac {Z_m^{\theta}}{n} +\frac{\overline{Z}_m^{\theta}}{\sqrt{n}}|\ge\delta\}}\rt]\\
\to&0,\quad \text{ as }n\to\infty\text{ and }\ep\to0,
\end{align*}
where the convergence is due to the finiteness of $\om$ and $\os$.

To prove (\ref{lemma-taylor1}),
by the remainder estimation in (\ref{remainder-esti}), we only need to prove
$$  \sup
\limits_{ \theta\in\Theta}E_P\lt[\Gamma(m,n, \theta) \rt]=\sup
\limits_{ \theta\in\Theta}E_P\left[L_{m,n}(T_{m-1,n}^{\theta} )\right], \;\forall n\ge m\ge 1.$$

By the common variance assumption, we have that for any $ \theta\in\Theta$,
$$E_P\lt[\overline{Z}_m^{\theta}|\CH^{\theta}_{m-1}\rt]=0,\; E_P\lt[\lt(\overline{Z}_m^{\theta}\rt)^2|\CH^{\theta}_{m-1}\rt]=1.$$

An application of  Lemma~\ref{tab-proper} leads to
\begin{align*}
&  \sup
\limits_{ \theta\in\Theta}E_P\lt[\Gamma(m,n, \theta) \rt]\\
=&\sup\limits_{ \theta\in\Theta}E_P\Bigg[H_{\frac{m}{n}}(T_{m-1,n}^{\theta} )+ \dot{H}_{\frac{m}{n}}(T_{m-1,n}^{\theta} )E_P\lt[\left(\frac {Z_m^{\theta}}{n}
+\frac{\overline{Z}_m^{\theta}}{\sqrt{n}}\right)|\CH^{\theta}_{m-1}\rt]\\
& \hspace{7.5cm}+\frac{1}{2} \ddot{H}_{\frac{m}{n}}(T_{m-1,n}^{\theta} )E_P\bigg[\bigg(
\frac{\overline{Z}_m^{\theta}}{\sqrt{n}}\bigg)^2|\CH^{\theta}_{m-1}\bigg] \Bigg]\\
=&\sup\limits_{ \theta\in\Theta}E_P\left[H_{\frac{m}{n}}(T_{m-1,n}^{\theta} )+ \frac {1}{n}\dot{H}_{\frac{m}{n}}(T_{m-1,n}^{\theta} )Z_m^{\theta}+\frac{1}{2n} \ddot{H}_{\frac{m}{n}}(T_{m-1,n}^{\theta} ) \right]\\
=&\sup\limits_{ \theta\in\Theta}E_P\bigg[H_{\frac{m}{n}}(T_{m-1,n}^{\theta} )+
\frac {1}{n}\lt(\dot{H}_{\frac{m}{n}}(T_{m-1,n}^{\theta} )\mu_L\vee \dot{H}_{\frac{m}{n}}(T_{m-1,n}^{\theta} )\mu_R\rt)
+\frac{1}{2n} \ddot{H}_{\frac{m}{n}}(T_{m-1,n}^{\theta} ) \bigg]\\
=&\sup\limits_{ \theta\in\Theta}E_P\Big[H_{\frac{m}{n}}(T_{m-1,n}^{\theta})+ \frac {\om}{n} \big(\dot{H}_{\frac{m}{n}}(T_{m-1,n}^{\theta})\big)^+\!\!-\frac {\um}{n} \big(\dot{H}_{\frac{m}{n}}(T_{m-1,n}^{\theta})\big)^-\!\!+\frac{1}{2n} \ddot{H}_{\frac{m}{n}}(T_{m-1,n}^{\theta})\Big]\\
=& \sup
\limits_{ \theta\in\Theta}E_P\left[L_{m,n}(T_{m-1,n}^{\theta} )\right],
\end{align*}
where the last equality holds due to $\om=-\um$ under assumption (\ref{means-condition}). This completes the proof.
\end{proof}

The following lemma is important for the proof of Theorem \ref{thm-CLT0}.
\begin{lemma}
\label{lemma-taylor} Define a family of functions   $\{\widehat{L}_{m,n}(x)\}_{m=1}^n$ by
\beq\label{function-hatLt}
\widehat{L}_{m,n}(x)=H_{\frac{m}{n}}(x)+\frac{\um}{n}\lt| \dot{H}_{\frac{m}{n}}(x)\rt|+\frac{1}{2n} \ddot{H}_{\frac{m}{n}}(x),\quad x\in\CR.
\eeq
Let $\{\theta^{n,c}: n\geq1\}$
be the strategies given in (\ref{strategy}),
then  the followings hold.
\begin{description}
\item[(1)] Under the assumption that $\mu_L\ge \mu_R$ in (\ref{musigma}),
    \begin{description}
    \item[(a)]
If $sgn(\dot{\ph}(x))=-sgn(x-c)$ for all $x\in\CR$, then
\begin{align}
\lim_{n\to\infty}\sum_{m=1}^n \left|E_{ P }\left[  H_{\frac{m}{n}}\left(  T_{m,n}^{\theta^{n,c}}\right)\right] - E_{ P }\left[L_{m,n}\lt( T_{m-1,n}^{\theta^{n,c}}\rt)\right]\right|=0,\label{sym-ty-1}
\end{align}\normalsize
\item[(b)] If  $sgn(\dot{\ph}(x))=sgn(x-c)$ for all $x\in\CR$, then
\begin{align}
\lim_{n\to\infty}\sum_{m=1}^n \left|E_{ P }\left[  H_{\frac{m}{n}}\left(  T_{m,n}^{\theta^{n,c}}\right)\right] - E_{ P }\left[\widehat{L}_{m,n}\lt( T_{m-1,n}^{\theta^{n,c}}\rt)\right]\right|=0,\label{sym-ty-2}
\end{align}
\end{description}
\item[(2)] Under the assumption that $\mu_L< \mu_R$ in (\ref{musigma}),
    \begin{description}
    \item[(c)]
If $sgn(\dot{\ph}(x))=-sgn(x-c)$ for all $x\in\CR$, then
\begin{align}
\lim_{n\to\infty}\sum_{m=1}^n \left|E_{ P }\left[  H_{\frac{m}{n}}\left(  T_{m,n}^{\theta^{n,c}}\right)\right] - E_{ P }\left[\widehat{L}_{m,n}\lt( T_{m-1,n}^{\theta^{n,c}}\rt)\right]\right|=0,\label{sym-ty-3}
\end{align}\normalsize
\item[(d)] If  $sgn(\dot{\ph}(x))=sgn(x-c)$ for all $x\in\CR$, then
\begin{align}
\lim_{n\to\infty}\sum_{m=1}^n \left|E_{ P }\left[  H_{\frac{m}{n}}\left(  T_{m,n}^{\theta^{n,c}}\right)\right] - E_{ P }\left[L_{m,n}\lt( T_{m-1,n}^{\theta^{n,c}}\rt)\right]\right|=0,\label{sym-ty-4}
\end{align}
\end{description}
\end{description}
\end{lemma}
\begin{proof} We only give the proof of (1)-(a), the rest of the proofs are similar.

We suppose that $\mu_L\ge \mu_R$ in (\ref{musigma}) and for any $x\in\CR$, $sgn(\dot{\ph}(x))=-sgn(x-c)$. Let  $\theta^{n,c}=(\vartheta^{n,c}_1,\cdots,\vartheta^{n,c}_{m},\cdots)$ be the strategy given in (\ref{strategy}). It follows from (3) in Lemma~\ref{lemma-ddp}   and  direct calculation that, for $1\le m\le n$,
\begin{align*}
&\quad\  E_{ P}\lt[\Gamma(m,n, \theta^{n,c}) \rt]\\
&=E_{ P}\left[H_{\frac{m}{n}}(T_{m-1,n}^{\theta^{n,c}})+ \dot{H}_{\frac{m}{n}}(T_{m-1,n}^{\theta^{n,c}}  )\left(\frac {Z_m^{\theta^{n,c}}}{n} +\frac {\overline{Z}_m^{\theta^{n,c}}}{\sqrt{n}}\right)+\frac{1}{2} \ddot{H}_{\frac{m}{n}}(T_{m-1,n}^{\theta^{n,c}}  ) \lt(\frac {\overline{Z}_m^{\theta^{n,c}}}{\sqrt{n}}\rt)^2\right]\\
&=E_{ P}\lt[H_{\frac{m}{n}}(T_{m-1,n}^{\theta^{n,c}}  )+ \dot{H}_{\frac{m}{n}}(T_{m-1,n}^{\theta^{n,c}} )E_{ P}\left[\left(\frac {Z_m^{\theta^{n,c}}}{n} +\frac {\overline{Z}_m^{\theta^{n,c}}}{\sqrt{n}}\right)|\CH^{\theta}_{m-1}\rt]\rt.\\
&\hspace{6.5cm}\lt.+ \frac{1}{2} \ddot{H}_{\frac{m}{n}}(T_{m-1,n}^{\theta^{n,c}}  ) E_{ P}\left[\lt(\frac {\overline{Z}_m^{\theta^{n,c}}}{\sqrt{n}}\rt)^2|\CH^{\theta}_{m-1}\rt]\rt]\\
&=E_{ P}\lt[H_{\frac{m}{n}}(T_{m-1,n}^{\theta^{n,c}}  )+ \frac{\om}{n}\dot{H}_{\frac{m}{n}}(T_{m-1,n}^{\theta^{n,c}}  )I_{\{\vartheta^{n,c}_{m}=1\}}+ \frac{\um}{n}\dot{H}_{\frac{m}{n}}(T_{m-1,n}^{\theta^{n,c}}  )I_{\{\vartheta^{n,c}_{m}=2\}}\rt.\\
&\lt.\hspace{10.5cm}+\frac{1}{2n } \ddot{H}_{\frac{m}{n}}(T_{m-1,n}^{\theta^{n,c}}  ) \right]\\
&=  E_{ P}\left[H_{\frac{m}{n}}(T_{m-1,n}^{\theta^{n,c}})+ \frac {\om}{n}\lt|\dot{H}_{\frac{m}{n}}(T_{m-1,n}^{\theta^{n,c}})\rt|+\frac{1}{2n} \ddot{H}_{\frac{m}{n}}(T_{m-1,n}^{\theta^{n,c}})\right]\\
&=  E_{ P}\left[L_{m,n}(T_{m-1,n}^{\theta^{n,c}}  )\right],
\end{align*}\normalsize
which combined with \rf{remainder-esti} implies  (\ref{sym-ty-1}) and the lemma.
\end{proof}

Now we are ready to prove Theorems~\ref{thm-CLT0}-\ref{thm-CLT1}. The main idea is to compare the individual terms in
$T_{n,n}^{\theta}$
to the  increments of the solution of SDE (\ref{sde}) over small intervals.

\begin{proof}[Proof of Theorem~\ref{thm-CLT0}] We only give the proof of  (1),  (2) can be proved similarly. For any fixed $c\in\CR$,  let $\ph\in C(\overline{\CR})$ be symmetric with centre $c\in\CR$. The result is clear if $\varphi$ is globally constant.  Thus we assume that $\varphi$ is not a constant function.

 Assume that  $\ph$ is decreasing on $(c,\infty)$ (the case that $\ph$ is increasing on $(c,\infty)$ can be proved similarly). For any  $h>0$, define the function $\ph_h$ by
\[
\varphi_h(x)=\int_{-\infty}^{\infty}\frac{1}{\sqrt{2\pi}}\varphi(x+hy)e^{-\tfrac{y^2}{2}}dy . \label{phh}
\]
By the Approximation Lemma in \cite[Ch. VIII]{feller}, we have that
\begin{align}
\lim_{h\to0}\sup_{x\in\CR}|\ph(x)-\ph_h(x)|=0.\label{approxi}
\end{align}
It follows from direct calculation that
\begin{align*}
\varphi_{h}(x+c)=  &  \int_{-\infty}^{\infty}\frac{1}{\sqrt{2\pi}}%
\varphi(x+c+hy)e^{-\tfrac{y^{2}}{2}}dy\\
=  &  \int_{-\infty}^{\infty}\frac{1}{\sqrt{2\pi}}\varphi(-x+c-hy)e^{-\tfrac
{y^{2}}{2}}dy\\
=  &  \int_{-\infty}^{\infty}\frac{1}{\sqrt{2\pi}}\varphi(-x+c+hy)e^{-\tfrac
{y^{2}}{2}}dy\\
=  &  \varphi_{h}(-x+c).
\end{align*}
Thus $\varphi_{h}$ is symmetric with centre $c$. In addition,  we have
\begin{align*}
\dot{\varphi}_{h}(x)= 
 &  \int_{-\infty}^{\infty}\frac{1}{\sqrt{2\pi}h^{3}} \varphi
(x+y)  ye^{-\tfrac{y^{2}}{2h^{2}}}dy\\
=  &  \int_{0}^{\infty}\frac{1}{\sqrt{2\pi}h^{3}} \varphi
(c+y+x-c)ye^{-\tfrac{y^{2}}{2h^{2}}}dy\\
&  +\int_{-\infty}^{0}\frac{1}{\sqrt{2\pi}h^{3}} \varphi
(c+y+x-c)  ye^{-\tfrac{y^{2}}{2h^{2}}}dy\\
=  &  \int_{0}^{\infty}\frac{1}{\sqrt{2\pi}h^{3}}\left(  \varphi
(c+y+x-c)-\varphi(c+y+c-x)\right)  ye^{-\tfrac{y^{2}}{2h^{2}}}dy.
\end{align*}
Since $\varphi$ is decreasing on $(c,\infty)$, it follows that
 $$sgn(\dot{\varphi}_{h} (x))=-sgn(x-c).$$

In the remaining proof of this theorem, we assume that $(\mu_l,\mu_R)=(\om,\um)$, and we first consider the case that $\om=-\um$, we continue to use $\{H_t(x)\}_{t\in[0,1]}$ to denote the functions defined in \rf{hf} with  $\ph_h$ in place of $\ph$  and $\alpha=\um$ there. Let $\{L_{m,n}(x)\}_{m=1}^n$ be functions defined in \rf{function-Lt} with $\{H_t(x)\}_{t\in[0,1]}$ here.

For a large enough $n$, let $\theta^{n,c}$ be the strategy defined in (\ref{strategy}), and let $\eta_0\sim \mathcal{B}(\um,0,c)$, by direct calculation we obtain
\begin{align*}
&\quad\ E_P\left[  \varphi_h\left(  T_{n,n}^{\theta^{n,c}}\right)  \right]
-E_P[\varphi_h\left(  \eta_0\right)]   \\
&= E_P\left[  H_{1}\left( T_{n,n}^{\theta^{n,c}}\right)  \right]  -H_{0}(0)\\
&=  E_P\left[  H_{1}\left( T_{n,n}^{\theta^{n,c}}\right)  \right]  -E_P\left[  H_{\frac{n-1}{n}}\left(  T_{n-1,n}^{\theta^{n,c}}\right)  \right] \\
&\quad\   + E_P\left[  H_{\frac{n-1}{n}}\left( T_{n-1,n}^{\theta^{n,c}}\right)  \right]  -E_P\left[  H_{\frac{n-2}{n}}\left(  T_{n-2,n}^{\theta^{n,c}}\right)  \right]  +\ldots\\
&\quad\   +E_P\left[  H_{\frac{m}{n}}\left(  T_{m,n}^{\theta^{n,c}}\right)  \right]  -E_P\left[  H_{\frac{m-1}{n}}\left( T_{m-1,n}^{\theta^{n,c}}\right)  \right]  +\ldots\\
& \quad\  + E_P\left[  H_{\frac{1}{n}}\left( T_{1,n}^{\theta^{n,c}}\right)  \right]  -H_{0}(T_{0,n}^{\theta^{n,c}})\\
&=  \sum\limits_{m=1}^{n}\left\{ E_P\left[
H_{\frac{m}{n}}\left(T_{m,n}^{\theta^{n,c}}\right)  \right]
-E_P\left[  H_{\frac{m-1}{n}}\left( T_{m-1,n}^{\theta^{n,c}}\right)  \right]  \right\} \\
&=  \sum\limits_{m=1}^{n}\left\{ E_P\left[
H_{\frac{m}{n}}\left(  T_{m,n}^{\theta^{n,c}} \right)  \right]
-E_P\left[  L_{m,n}\left( T_{m-1,n}^{\theta^{n,c}}\right)  \right]  \right\} \\
&\quad\   +\sum_{m=1}^{n}\left\{ E_P\left[
L_{m,n}\left(  T_{m-1,n}^{\theta^{n,c}}\right)
\right]  -E_P\left[  H_{\frac{m-1}{n}}\left(
T_{m-1,n}^{\theta^{n,c}}\right)  \right]  \right\} \\
&= : I_{1n}+I_{2n}\text{.}%
\end{align*}
An  application of  Lemma~\ref{lemma-taylor} implies that $|I_{1n}|\to0$ as $n\to\infty.$
It follows from (5) in Lemma~\ref{lemma-ddp} that
\begin{align*}
|I_{2n}|  
&  \leq\sum_{m=1}^{n}\sup\limits_{x\in\mathbb{R}}\left\vert L_{m,n}%
(x)-H_{\frac{m-1}{n}}(x)\right\vert \\
&  =\sum_{m=1}^{n}\sup\limits_{x\in\mathbb{R}}\left\vert H_{\frac{m-1}{n}}(x)- H_{\frac{m}{n}}\left(
x\right)  -\frac{\om}{n}\left|  \dot{H}_{\frac{m}{n}}(x)\right|  -\frac{1}%
{2n}\ddot{H}_{\frac{m}{n}}(x)  \right\vert \\
&\to0, \text{ ~as } n\ra \infty,
\end{align*}
which implies that
\begin{equation}
\lim_{h\rightarrow0}\lim_{n\rightarrow\infty}\left\vert
E_P\left[  \varphi_h\left(  T_{n,n}^{\theta^{n,c}}\right)  \right] -
E_P[\varphi_h\left(  \eta_0\right)] \right\vert =0\text{.}\label{eq3-1}%
\end{equation}

 Putting together \rf{approxi} and \rf{eq3-1}, we have
\begin{align*}
&  \lim_{n\rightarrow\infty}\left\vert E_P\left[  \varphi\left( T_{n,n}^{\theta^{n,c}}\right)
\right]  - E_P[\varphi\left(  \eta_0\right)]
\right\vert \\
\leq & \lim_{h\rightarrow0} \lim_{n\rightarrow\infty}\left\vert E_P\left[  \varphi\left(  T_{n,n}^{\theta^{n,c}}\right)
\right]  -E_P\left[  \varphi_h\left(  T_{n,n}^{\theta^{n,c}}\right)  \right]
\right\vert \\
&+\lim_{h\rightarrow0}\lim_{n\rightarrow\infty}\left\vert
E_P\left[  \varphi_h\left(  T_{n,n}^{\theta^{n,c}}\right)  \right] - E_P[\varphi_h\left(  \eta_0\right)]
\right\vert \\
&+\lim_{h\rightarrow0} \left\vert
E_P[\varphi_h\left(  \eta_0\right)]-E_P[\varphi\left(  \eta_0\right)]
\right\vert  \\
=&0\text{.}%
\end{align*}
Finally, we describe the proof for the general $\om$ and $\um$. For any $\theta\in\Theta$, let $Y_i^{\theta}=Z_i^{\theta}-\frac{\om+\um}{2}$, and then
$$
\esssup_{\theta\in\Theta}E_P[Y_i^{\theta}|\CH^{\theta}_{i-1}]=\frac{\om-\um}{2},\quad
\essinf_{\theta\in\Theta}E_P[Y_i^{\theta}|\CH^{\theta}_{i-1}]=-\frac{\om-\um}{2}.
$$
It can be checked that,
\begin{align*}
&  \lim_{n\rightarrow\infty}E_P\left[  \varphi\left( T_{n,n}^{\theta^{n,c}}\right)
\right] \\
=&  \lim_{n\rightarrow\infty} E_P\left[  \varphi\left( \frac{1}{n}\sum_{i=1}^nZ_i^{\theta^{n,c}}+\frac{1}{\sqrt{n}}\sum_{i=1}^n\frac{Z_i^{\theta}-E_P[Z_i^{\theta^{n,c}}|\CH^{\theta^{n,c}}_{i-1}]}{\sigma}\right)
\right] \\
=&  \lim_{n\rightarrow\infty}  E_P\left[  \varphi\left( \frac{\om+\um}{2}+ \frac{1}{n}\sum_{i=1}^nY_i^{\theta^{n,c}}+\frac{1}{\sqrt{n}}\sum_{i=1}^n\frac{Y_i^{\theta^{n,c}}-E_P[Y_i^{\theta^{n,c}}|\CH^{\theta^{n,c}}_{i-1}]}{\sigma}\right)
\right] \\
=&  \lim_{n\rightarrow\infty}E_P\left[  \hat{\varphi}\left(  \frac{1}{n}\sum_{i=1}^nY_i^{\theta^{n,c}}+\frac{1}{\sqrt{n}}\sum_{i=1}^n\frac{Y_i^{\theta^{n,c}}-E_P[Y_i^{\theta^{n,c}}|\CH^{\theta^{n,c}}_{i-1}]}{\sigma}\right)
\right],
\end{align*}
where $\hat{\ph}(x)=\ph(x+\frac{\om+\um}{2})$.
Since the strategy $\theta^{n,c}$ can be also rewrite in the following forms
\begin{align*}
\vartheta_m^{n,c}=&2-I_{\{T_{m-1,n}^{\theta^{n,c}}\le c-(1-\frac{m-1}{n})\frac{\om+\um}{2}\}}\\
=&2-I_{\left\{\frac{1}{n}\sum_{i=1}^{m-1}Y_i^{\theta^{n,c}}+\frac{1}{\sqrt{n}}\sum_{i=1}^{m-1}\frac{Y_i^{\theta^{n,c}}-E_P[Y_i^{\theta^{n,c}}|\CH^{\theta^{n,c}}_{i-1}]}{\sigma}\le c-\frac{\om+\um}{2}\right\}}.
\end{align*}
Apply the above results for $\{Y_i^{\theta^{n,c}}:i\ge1\}$, we have
\begin{align*}
&\lim_{n\rightarrow\infty}E_P\left[  \varphi\left( T_{n,n}^{\theta^{n,c}}\right)\right]\\
=&\lim_{n\rightarrow\infty}E_P\left[  \hat{\varphi}\left(  \frac{1}{n}\sum_{i=1}^nY_i^{\theta^{n,c}}+\frac{1}{\sqrt{n}}\sum_{i=1}^n\frac{Y_i^{\theta^{n,c}}-E_P[Y_i^{\theta^{n,c}}|\CH^{\theta^{n,c}}_{i-1}]}{\sigma}\right)
\right]\\
=&E_P[\hat{\varphi}\left(  \eta_0'\right)]= \int_{\CR} \hat{\varphi}\left(  y\right)f^{\frac{\um-\om}{2},0,c-\frac{\om+\um}{2}}(y)dy
= \int_{\CR}  \varphi\left(  y\right)f^{\frac{\um-\om}{2},0,c-\frac{\om+\um}{2}}(y-\tfrac{\om+\um}{2})dy\\
=&\int_{\CR}  \varphi\left(  y\right)f^{\frac{\um-\om}{2},\frac{\om+\um}{2},c}(y)dy=E_P[\varphi\left(  \eta_1\right)],
\end{align*}
where $\eta_0'\sim\mathcal{B}(\frac{\um-\om}{2},0,c-\frac{\om+\um}{2})$ and $\eta_1\sim\mathcal{B}(\frac{\um-\om}{2},\frac{\om+\um}{2},c)$. Then we complete the proof.
\end{proof}
\begin{proof}[Proof of Theorem~\ref{thm-CLT1}]
The proof  follows from  (\ref{lemma-taylor1}) in Lemma~\ref{remainder-lemma}, and similar arguments used in the proof of Theorem~\ref{thm-CLT0}.
\end{proof}
\begin{proof}[Proof of Corollary~\ref{sym-max}]
We still prove the result for $\om=-\um$ firstly and then for the general $\om$ and $\um$.

(1) follows directly from Theorem~\ref{thm-CLT0}.

To prove (2),  for any large enough $n$, let  $\{H_{\frac{m}{n}}(x)\}_{m=1}^n$ be functions defined by (\ref{hf}) with $\ph(x)$ replaced by $\ph(\hat{\sigma} x)$ and $(Y_1^{t,x,\alpha,c})$ replaced by $(Y_1^{t,\frac{x}{\hat{\sigma}},\hat{\alpha}_n,\frac{c}{\hat{\sigma}}})$. Similar as the proof of Lemma~\ref{lemma-ddp}-(5) and Lemma~\ref{lemma-taylor}-(1), we can prove
\begin{align}
&\lim_{n\to\infty}\sum_{m=1}^n \left|E_{ P }\left[  H_{\frac{m}{n}}\left(  \hat{T}_{m,n}^{\hat{\theta}^{n,c}}\right)\right] - E_{ P }\left[L_{m,n}^*\lt( \hat{T}_{m-1,n}^{\hat{\theta}^{n,c}}\rt)\right]\right|=0,\label{H1-ty-1}\\
&\lim_{n\to\infty }\sum_{m=1}^{n}\sup\limits_{x\in\mathbb{R}}\left\vert H_{\frac{m-1}{n}}\left(  x\right)
-L_{m,n}^*(x)\right\vert =0,\label{H1-ty-2}
\end{align}
where $L_{m,n}^*(x)=H_{\frac{m}{n}}(x)-\frac {\hat{\alpha}_n}{n}\lt| \dot{H}_{\frac{m}{n}}(x)\rt|+\frac{\hat{\sigma}^2}{2n} \ddot{H}_{\frac{m}{n}}(x)$.
With  the similar arguments in the proof of   Theorem~\ref{thm-CLT0}, we obtain the result.
\end{proof}

\subsection{Proof of Large Deviation Principles}
\begin{proof}[Proof of Theorem~\ref{LDP}] Recall the functions $I(x)$, $\Lambda^*_{\om}(\lambda)$ and  $\Lambda^*_{\um}(x)$ defined in \rf{ratefunction} and \rf{Lambda-star-om-um}.

 Next we establish the upper estimate \rf{ldp1}. Let $F$ be a closed set in $\CR$.

 \vspace{3mm}

 If $F\cap [\um,\om] \neq \emptyset$, then $\inf_{x\in F}I(x)=0$
 and the upper estimate holds.

\vspace{3mm}


 For $F\cap [\um,\om] =\emptyset$, 
 there are two possibilities.

\vspace{3mm}

 First consider the case $\om<y=\inf\{x\in F\}$. By direct calculation, we have for any $\lambda\geq 0$ and $\theta$ in $\Theta$
\beq
 && E_P\lt[I_{F}\lt(\frac{S^{\theta}_n}{n}\rt)\rt]\leq E_P[e^{\lambda(S^{\theta}_n-ny)}] \\
  &&=e^{-n\lambda y} E_{P}\bigg[e^{\lambda S^{\theta}_{n-1}}E_P\lt[e^{\lambda Z_n^{\theta}}|\CH^{\theta}_{n-1}\rt]\bigg]\\
  && =e^{-n\lambda y} E_{P}\bigg[ e^{\lambda S^{\theta}_{n-1}}E_P\lt[\lt(I_{\{\vartheta_n=1\}}e^{\lambda W_n^{L}}+ I_{\{\vartheta_n=2\}}e^{\lambda W_n^R}\rt)|\CH^{\theta}_{n-1}\rt]\bigg]\\
&& \leq e^{-n\lambda y} E_{P}\lt[ e^{\lambda S^{\theta}_{n-1}}e^{\Lambda_{\om}(\lambda)}\rt]  \\
&& \leq e^{-n\lt(\lambda y- \Lambda_{\om}(\lambda)\rt)},
  \eeq
where in the second last inequality we used the fact that $\Lambda_{\om}(\lambda)=\max\{\Lambda_{\mu_L}(\lambda),\Lambda_{\mu_R}(\lambda)\}$ for $\lambda \geq 0$.

Taking the supremum over $ \theta\in\Theta$ and $\lambda \geq 0$, we obtain that
\[
\frac{1}{n}\log \sup_{ \theta\in\Theta}E_P\lt[I_{F}\lt(\frac{S^{\theta}_n}{n}\rt)\rt]\leq -\sup_{\lambda \geq 0}\{\lambda y -\Lambda_{\om}(\lambda)\}=-\Lambda^{\ast}_{\om}(y).\]
where the equality holds due to the fact that $y> \om.$
On the other hand, the function $\Lambda^{\ast}_{\om}(x)$ is non-decreasing for $x \geq \om$. Thus $ \Lambda^{\ast}_{\om}(y)=\inf_{x\in F}I(x)$
 and
\[
\limsup_{n\ra \infty}\frac{1}{n}\log \nu_n(F)\leq -\inf_{x\in F}I(x).\]

If $\um>y=\sup\{x\in F\}$, then we have for $\lambda <0,$
\beq
 && E_P\lt[I_{F}\lt(\frac{S^{\theta}_n}{n}\rt)\rt]\leq E_P[e^{\lambda(S^{\theta}_n-ny)}] \\
  &&=e^{-n\lambda y} E_{P}\bigg[e^{\lambda S^{\theta}_{n-1}}E_P\lt[e^{\lambda Z_n^{\theta}}|\CH^{\theta}_{n-1}\rt]\bigg]\\
  && =e^{-n\lambda y} E_{P}\bigg[ e^{\lambda S^{\theta}_{n-1}}E_P\lt[\lt(I_{\{\vartheta_n=1\}} e^{\lambda W_n^{L}}+ I_{\{\vartheta_n=2\}}e^{\lambda W_n^R}\rt)|\CH^{\theta}_{n-1}\rt]\bigg]\\
&& \leq e^{-n\lambda y} E_{P}\lt[ e^{\lambda S^{\theta}_{n-1}}e^{\Lambda_{\um}(\lambda)}\rt]  \\
&& \leq e^{-n\lt(\lambda y- \Lambda_{\um}(\lambda)\rt)},
  \eeq
  where we used the fact that $\Lambda_{\um}(\lambda)=\max\{\Lambda_{\mu_L}(\lambda),\Lambda_{\mu_R}(\lambda)\}$ for $\lambda < 0$. Taking the supremum over $ \theta\in\Theta$ and $\lambda <0$, we obtain that
\[
\frac{1}{n}\log \sup_{ \theta\in\Theta}E_P\lt[I_{F}\lt(\frac{S^{\theta}_n}{n}\rt)\rt]\leq -\sup_{\lambda < 0}\{\lambda y -\Lambda_{\um}(\lambda)\}=-\Lambda^{\ast}_{\um}(y).\]
where the equality holds due to the fact that $y< \om.$
  Noting that
$$
\Lambda^{\ast}_{\um}(y)=\inf_{x\in F}I(x),
$$
 it follows that \rf{ldp1} also holds in this case. Putting all these together we obtain the upper estimate.

Next we turn to the proof of the lower estimate \rf{ldp2}.  For any open set $G$ in $\CR$, the lower estimate holds trivially if $\inf_{x\in G}I(x)=\infty$. Next assume that $\inf_{x\in G}I(x)<\infty$.
  For any $0\leq \alpha \leq 1$, construct a strategy $\theta^{\alpha}$ as follows.

  Step 1: Choosing $\vartheta_1^{\alpha}=1, \vartheta_2^{\alpha}=2$.

  Step 2: $\vartheta_3^{\alpha}=1$ if $1/2< \alpha$. It is $2$ otherwise.

  Step 3: For any $n\geq 4$, let $m_{n-1}$ be the number of times that $1$ appears among $\{\vartheta_1^{\alpha}, \ldots, \vartheta_{n-1}^{\alpha}\}$.  Then $\vartheta_n^{\alpha}=1$ if $\frac{m_{n-1}}{n-1}< \alpha$, and $\vartheta_n^{\alpha}=2$ otherwise.

  It follows from the construction that
\beq
&&\lim_{n\ra \infty} \frac{1}{n}\log E_{P}[e^{\lambda S^{\theta^{\alpha}}_n}]\\
&&=\lim_{n\ra \infty} \frac{1}{n}\log \bigg(E_{P}[e^{\lambda W^L}]\bigg)^{m_n}\bigg(E_{P}[e^{\lambda W^R}]\bigg)^{n-m_n}\\
&& = \alpha \Lambda_{\mu_L}(\lambda)+ (1-\alpha) \Lambda_{\mu_R}(\lambda).
\eeq

By Cram\'er theorem (\cite{DZ98}), the family $\{P\circ (\tfrac{S^{\theta^{\alpha}}_n}{n})^{-1}\}$ satisfies a large deviation principle with speed $n$ and rate function
\[
I_{\alpha}(x)=\sup_{\lambda\in\CR}\{
\lambda x-\alpha \Lambda_{\mu_L}(\lambda)- (1-\alpha) \Lambda_{\mu_R}(\lambda) \}.
\]

Noting that
$S^{\theta^{\alpha}}_n/n$ has the same distribution as $$ \frac{m_n}{n}\frac{1}{m_n}\sum_{j=1}^{m_n}W_j^L+\frac{n-m_n}{n}\frac{1}{n-m_n}\sum_{j=1}^{n-m_n}W_j^R.$$
Since $\lim_{n\ra \infty}\frac{m_n}{n}=\alpha$, it follows that $\frac{1}{m_n}\sum_{j=1}^{m_n}W_j^L$ and $\frac{1}{n-m_n}\sum_{j=1}^{n-m_n}W_j^R$ satisfy large deviation principles with the same speed $n$ and respective rate functions $\alpha \Lambda^{\ast}_{\mu_L}(\cdot)$ and $(1-\alpha)\Lambda^{\ast}_{\mu_R}(\cdot)$, where
for $ x \in \mathbb{R}$
 \beq
 \Lambda^{\ast}_{\mu_R}(x)&=&\sup_{\lambda \in \mathbb{R}}\{\lambda x- \Lambda_{\mu_R}(\lambda)\}, \\
 \Lambda^{\ast}_{\mu_L}(x)&=&\sup_{\lambda \in \mathbb{R}}\{\lambda x- \Lambda_{\mu_L}(\lambda)\}.
 \eeq

Applying the contraction principle we obtain
 \[
 I_{\alpha}(x)=\inf\{ \alpha\Lambda^{\ast}_{\mu_L}(y)+(1-\alpha)\Lambda^{\ast}_{\mu_R}(z): \alpha y+(1-\alpha)z =x\},
 \]
which implies that, for $\hat{x}=\alpha \mu_L+(1-\alpha)\mu_R$, we $I_{\alpha}(\hat{x})=0$ by choosing $y=\mu_L, z=\mu_R$.

 By the definition of nonlinear probability we obtain that
\beq\label{lower1}
&&\liminf_{n\ra \infty} \frac{1}{n}\log\sup_{\theta\in\Theta}E_P\lt[I_G\lt(\frac{S^{\theta}_n}{n}\rt)\rt]\\
&& \ \ \ \  \geq \liminf_{n\ra \infty} \frac{1}{n}\log\sup_{\alpha \in [0,1]}E_P\lt[I_G\lt(\frac{S^{\theta^{\alpha}}_n}{n}\rt)\rt]\\
&&\ \ \ \  \geq -\inf_{x\in G}\inf_{\alpha \in [0,1]} I_{\alpha}(x).
\eeq

 To get the right lower estimate, we consider the open set $G$ in separate cases.

 First we assume that $G \cap [\um,\om]\neq \emptyset $. In this case we have
 \[
\inf_{x\in G}I(x)= \inf_{x\in G}\inf_{\alpha \in [0,1]} I_{\alpha}(x)=0.\]

 Next assume that  $G \cap [\um,\om]= \emptyset $. Set
 \[
 G_1 = G\cap (-\infty, \um), \ G_2 =G\cap (\om,+\infty).
 \]

By choosing either  $\alpha =0$ or $1$, we get that
\[
\inf_{x\in G_1}\Lambda^{\ast}_{\um}(x) \geq \inf_{x\in G}\inf_{\alpha \in [0,1]} I_{\alpha}(x)\]
and
\[
\inf_{x\in G_2}\Lambda^{\ast}_{\om}(x) \geq \inf_{x\in G}\inf_{\alpha \in [0,1]} I_{\alpha}(x).\]

Since
 \[
 \inf_{x\in G}I(x)=\min\{\inf_{x\in G_1}\Lambda^{\ast}_{\um}(x), \inf_{x\in G_2}\Lambda^{\ast}_{\om}(x)\},\]
it follows that the lower estimate holds in this case.

Putting all these together we obtain the lower estimate and thus the theorem.
\end{proof}

 \section*{Acknowledgements} The first author gratefully
acknowledges the support of the National Key R\&D Program of China (grant No.
2018YFA0703900) and  Taishan Scholars Project (grant No. ZR2019ZD41). Shui Feng's research is supported by the Natural Sciences and Engineering Research Council of Canada.
Guodong Zhang's research is supported by the Shandong Provincial Natural Science Foundation, China (grant No. ZR2021MA098).

\end{document}